\documentclass[a4paper,11pt]{amsart}
\usepackage[top=2.2cm,left=2.8cm,right=2.8cm,bottom=2.8cm]{geometry}

\usepackage[foot]{amsaddr}

\usepackage{hyperref}
\hypersetup{
    colorlinks,
    citecolor=green,
    filecolor=black,
    linkcolor=blue,
    urlcolor=blue
}
\hypersetup{linktocpage}

\usepackage{mathtools}
\usepackage{cite}
\usepackage{graphicx}
\usepackage{subcaption}
\usepackage{cleveref}
\usepackage{amsbsy}
\usepackage{latexsym}
\usepackage{amsfonts}
\usepackage{amssymb}
\usepackage{upgreek}
\usepackage[usenames]{color}
\usepackage{amsmath,amsthm}
\usepackage{enumerate}

\usepackage{stmaryrd}
\usepackage{dsfont}

\DeclareMathOperator{\supp}{supp}
\DeclareMathOperator{\diam}{diam}
\DeclareMathOperator{\meas}{meas}
\DeclareMathOperator{\Span}{span}
\DeclareMathOperator*{\einf}{ess\,inf}

\DeclareMathOperator{\level}{level}

\providecommand{\abs}[1]{\lvert#1\rvert}
\providecommand{\bigabs}[1]{\bigl\lvert#1\bigr\rvert}
\providecommand{\Bigabs}[1]{\Bigl\lvert#1\Bigr\rvert}

\providecommand{\norm}[1]{\lVert#1\rVert}
\providecommand{\bignorm}[1]{\bigl\lVert#1\bigr\rVert}
\providecommand{\Bignorm}[1]{\Bigl\lVert#1\Bigr\rVert}
\providecommand{\biggnorm}[1]{\biggl\lVert#1\biggr\rVert}

\providecommand{\ceil}[1]{\lceil#1\rceil}

\newtheorem{theorem}{Theorem}
\newtheorem{lemma}[theorem]{Lemma}
\newtheorem{prop}[theorem]{Proposition}

\theoremstyle{definition}

\theoremstyle{remark}
\newtheorem{remark}[theorem]{Remark}

\numberwithin{equation}{section}
\numberwithin{theorem}{section}

\theoremstyle{plain}
\newtheorem{assumption}{Assumption}

\newcommand{\cA}{{\mathcal{A}}}

\newcommand{\cM}{{\mathcal{M}}}

\newcommand{\cF}{{\mathcal{F}}}

\newcommand{\cV}{\mathcal{V}}
\newcommand{\cT}{\mathcal{T}}

\newcommand{\bA}{\mathbf{A}}

\newcommand{\bP}{\mathbf{P}}

\newcommand{\bM}{\mathbf{M}}

\newcommand{\bu}{\mathbf{u}}
\newcommand{\bv}{\mathbf{v}}
\newcommand{\bg}{\mathbf{g}}
\newcommand{\bh}{\mathbf{h}}

\newcommand{\bz}{\mathbf{z}}

\newcommand{\bw}{\mathbf{w}}
\newcommand{\bbf}{\mathbf{f}}
\newcommand{\bB}{\mathbf{B}}
\newcommand{\br}{\mathbf{r}}
\newcommand{\bq}{\mathbf{q}}

\newcommand{\sdd}{\,\mathrm{d}}

\newcommand{\PP}{\mathbb{P}}

\newcommand{\Chi}{\raise .3ex
\hbox{\large $\chi$}} 

\newcommand{\T}{{\mathbb{T}}}
\newcommand{\R}{\mathbb{R}}

\newcommand{\N}{\mathbb{N}}

\newcommand{\msx}[1]{\text{\textsc{Mesh}}(#1)}

\newcommand{\initmesh}{\hat\cT_0}

\usepackage[section]{algorithm}
\usepackage{algorithmicx}
\usepackage{algpseudocode}

\title[A convergent adaptive finite element stochastic Galerkin method]{A convergent adaptive finite element stochastic Galerkin method based on multilevel expansions of random fields}

\author{Markus Bachmayr$^1$} 
\email{bachmayr@igpm.rwth-aachen.de}
\author{Martin Eigel$^2$}
\email{martin.eigel@wias-berlin.de}
\author{Henrik Eisenmann$^1$}
\email{eisenmann@igpm.rwth-aachen.de}
\author{Igor Voulis$^3$}
\email{i.voulis@math.uni-goettingen.de}
\address{$^1$ Institut f\"ur Geometrie und Praktische Mathematik, RWTH Aachen University, Templergraben 55, 52062 Aachen, Germany}
\address{$^2$ Weierstrass Institute for Applied Analysis and Stochastics, Mohrenstr.\ 39, 10117 Berlin, Germany}
\address{$^3$ Institute for Numerical and Applied Mathematics, University of G\"ottingen, Lotzestr.\ 16-18, 37083 G\"ottingen, Germany}
\date{\today}

\thanks{M.B.\ acknowledges funding by Deutsche Forschungsgemeinschaft (DFG, German Research Foundation) -- project numbers 442047500, 501389786. M.E.\ acknowledges funding by the DFG SPP 2298 and ANR-DFG project ``COFNET''. The work of H.E.~was funded by Deutsche Forschungs\-gemeinschaft – project number 501389786. I.V.\ acknowledges funding by the DFG through grant 432680300 -- SFB 1456.}

\date{\today}

\begin{document}

\maketitle

\begin{abstract}
  The subject of this work is an adaptive stochastic Galerkin finite element method for parametric or random elliptic partial differential equations, which generates sparse product polynomial expansions with respect to the parametric variables of solutions. For the corresponding spatial approximations, an independently refined finite element mesh is used for each polynomial coefficient. The method relies on multilevel expansions of input random fields and achieves error reduction with uniform rate. In particular, the saturation property for the refinement process is ensured by the algorithm. The results are illustrated by numerical experiments, including cases with random fields of low regularity.

\smallskip


\end{abstract}

\section{Introduction}

Elliptic partial differential equations with coefficients depending on countably many parameters arise in particular in problems of uncertainty quantification. In this context, they result from expansions of random fields on the computational domain as function series with scalar random coefficients corresponding to the parametric variables. The method constructed and analyzed here yields approximations of the parameter-dependent solutions using adaptive finite elements in the spatial variables combined with a sparse polynomial expansion in the parametric variables.

\subsection{Motivation: highly sparse representations and convergence of solvers}
For a class of ellliptic diffusion problems on a domain $D$, we aim to approximate the mapping from diffusion coefficients $a$ to the corresponding solutions $u$. Here, the first step is the choice of an expansion of $a$ as a function series $a(y) = \theta_i + \sum_{i \in \N} y_i \theta_i$ with $\theta_i \in L_\infty(D)$, $i \in \N_0$, parameterized by a countable family of scalar coefficients $y = (y_i)_{i \in \N}$. Since the function-valued mapping $y \mapsto u(y)$ is smooth with respect to these parameters, it is natural to approximate it in terms of suitable product polynomials $P_\nu(y) = \prod_i P_{\nu_i} (y_i)$ with $\deg P_{\nu_i} = \nu_i$ in the form $u(y) = \sum_{\nu \in F} u_\nu P_\nu(y)$, where $F$ is a subset of the finitely supported multi-indices in $\N_0^\N$. Under appropriate assumptions on the decay of $\norm{\theta_i}_{L_\infty}$ as $i\to \infty$, sparse selections of polynomial degrees $F$ lead to convergence bounds in terms of $\#F$ without a dependence on the number of active parameters \cite{CDS:11,BCM:17}.

Note that the coefficients $u_\nu$ in such polynomial expansions are still functions on $D$. Their approximability depends on the parameterization chosen for $a$: in particular, when $(\theta_i)_{i \in \N}$ is a dictionary with wavelet-type multilevel structure, the $u_\nu$ each have distinctive localized features that make them particularly amenable to finite element approximations with meshes adapted individually for each $\nu \in F$ \cite{BCDS:17}. Such adaptive discretizations can be realized based on stochastic Galerkin formulations, that is, joint variational formulations in terms of spatial and parametric variables. With spatial discretizations by adaptive selections from wavelet bases for each $u_\nu$, convergent algorithms of optimal computational complexity have been obtained in \cite{BV}.

However, the common case of adaptive finite element discretizations in space considered in this work poses additional difficulties when allowing independent meshes for each $u_\nu$. Since a stochastic Galerkin discretization leads to an elliptic system of PDEs (for $\#F$ solution components), the associated residuals generally contain edge functionals on various levels of refinement due to the interactions of different meshes. This means that a single step of a standard solve-mark-refine cycle may not necessarily ensure a uniform error reduction \cite{CDN:12}. Moreover, since Galerkin orthogonality cannot be ensured practically in this setting, standard residual error estimation techniques are not applicable.

We circumvent both of these issues by an error estimation strategy based on finite element frame coefficients of residuals (used in different contexts in \cite{HS:16,HSS:08}), which does not require Galerkin orthogonality, can naturally handle edge functional terms on different mesh levels, and allows for several refinements of a given element in one step of the adaptive scheme where required. We prove error reduction by a uniform factor in each step of the resulting scheme, which is also known as \emph{saturation property}. This is the first result of this kind since in the existing works (for example, \cite{CPB:19} and~\cite{BPR21}) on adaptive stochastic Galerkin methods with independent adaptive meshes for each $u_\nu$, this property was instead only \emph{assumed} to hold. We consequently prove convergence of the derived adaptive method.

We would like to stress that this convergence result by itself does \emph{not} depend on $(\theta_i)_{i \in \N}$ in the expansion of $a$ to exhibit a multilevel structure; it could be obtained in the same way for common choices of $\theta_i$ that are globally oscillating on $D$. However, we focus on the case of wavelet-like multilevel expansions of $a$ because such expansions are a natural choice, leading to favorable complexity bounds for each component of the adaptive scheme that otherwise are not attainable. Exploiting this type of multilevel structure, we show that the costs for each routine in the adaptive scheme have (up to logarithmic factors) optimal scaling. In our numerical tests, we observe the method to yield approximations converging at the expected optimal rate, with its runtime costs following the same optimal rate.

\subsection{Problem statement}

On a polygonal domain $D \subset \R^d$, where typically $d \in \{1,2,3\}$, we consider the elliptic model problem  
\begin{equation}\label{eq:diffusionequation}
   - \nabla \cdot ( a \nabla u ) = f \quad \text{on $D$,} \qquad u = 0 \quad \text{on $\partial D$,}
\end{equation}
in weak formulation with $f \in L_2(D)$.  Assuming $\cM_0$ to be a countable index set with $0 \in \cM_0$ and taking $\cM = \cM_0 \setminus \{ 0 \}$, the parameter-dependent coefficient $a$ is assumed to be given by an affine parameterization
\begin{equation} \label{eq:affinecoeff}
a(y) = \theta_0 + \sum_{\mu\in \cM} y_\mu \theta_\mu, \quad y = (y_\mu)_{\mu\in\cM} \in Y = [-1,1]^\cM
\end{equation}
 with $\theta_\mu \in L_\infty(D)$ for $\mu \in \cM_0$ and $\einf_D \theta_0 > 0$. Well-posedness is ensured by the \emph{uniform ellipticity condition} \cite{CDS:11}
\begin{equation}\label{eq:uniformellipticity} 
   c_B = \einf_D \biggl\{  \theta_0 - \sum_{\mu \in \cM} \abs{\theta_\mu} \biggr\} > 0.
\end{equation}
Note that as a consequence, for $C_B = \sup_{y \in Y} \norm{ a(y) }_{L_\infty(D)}$ we have
\begin{equation}\label{eq:CB}
  C_B  \leq 2 \norm{\theta_0}_{L_\infty} - c_B.
\end{equation}
To fix a probability distribution of the random coefficients $y \in Y$ 
in~\eqref{eq:affinecoeff}, we now introduce a product measure $\sigma$ on $Y$. For simplicity, we take $\sigma$ to be the uniform measure on $Y$, where $\sigma = \bigotimes_{\mu\in\cM} \sigma_1$ with $\sigma_1$ the uniform measure on $[-1,1]$. With $V = H^1_0(D)$, for each given $y \in Y$ let $u(y) \in V$ be defined by
\begin{equation}\label{eq:pwweakform}
  \int_D a(y) \nabla u(y)\cdot \nabla v \sdd x = \int_D f v \sdd x\quad\text{for all $v \in V$.}
\end{equation}
Then by~\eqref{eq:uniformellipticity}, the mapping 
$  Y \ni y\mapsto u(y) \in V $
can be regarded as an element of the Bochner space
\[
  \cV := L_2(Y, V, \sigma) \simeq V \otimes L_2(Y,\sigma).
\]

From the univariate Legendre polynomials $\{ L_k \}_{k\in\N}$ that are orthonormal with respect to the uniform measure on $[-1,1]$, we obtain the orthonormal basis $\{ L_\nu \}_{\nu\in \cF}$ of product Legendre polynomials for $L_2(Y,\sigma)$, which for $y\in Y$ are given by
\[
L_\nu(y) = \prod_{\mu \in \cM} L_{\nu_\mu}(y_\mu) ,\quad \nu\in\cF = \{ \nu \in \N_0^\cM \colon \text{$\nu_\mu \neq  0$ for finitely many $\mu \in \cM$}\} ;
\]
see, for example, \cite{Schwab:11,CD}. For $u \in \cV$ solving~\eqref{eq:diffusionequation} with coefficient~\eqref{eq:affinecoeff} in the sense of~\eqref{eq:pwweakform}, we thus have the basis expansion
\begin{equation}\label{eq:legendreexpansion}
   u(y) = \sum_{\nu \in \cF} u_\nu L_\nu(y), \quad u_\nu = \int_Y u(y)\,L_\nu(y)\sdd\sigma(y) \in V .
\end{equation}
By restricting the summation over $\nu$ in~\eqref{eq:legendreexpansion} to a finite subset $F \subset \cF$, we obtain \emph{semidiscrete} best approximations by elements of $V \otimes \Span \{ L_\nu \}_{\nu \in F}$. To obtain fully discrete computable approximations, for each $\nu$ the coefficient $u_\nu \in V$ needs to be replaced by a further approximation in some finite-dimensional subspace $V_\nu \subset V$. We thus aim to find an approximation of $u$ from a subspace
\begin{equation}\label{eq:subspacecV}
   \tilde\cV = \biggl\{ \sum_{\nu \in F} v_\nu L_\nu \colon v_\nu \in V_\nu, \nu \in F\biggr\} \subset \cV  
\end{equation}
of total dimension $\sum_{\nu \in F} \dim V_\nu$. In the present work, each $V_\nu$ is chosen as a suitable finite element subspace of $V$.

\subsection{Parameter expansions and approximability}

Different types of expansion~\eqref{eq:affinecoeff} can be used for the parametrized coefficient $a$. A typical choice are expansions with similar properties as Karhunen-Lo\`eve representations of random fields, where the functions $\theta_\mu$, $\mu \in \mathcal M$, oscillate with increasing frequencies on all of $D$. In the case $d=1$ with $D = (0,1)$, a popular test case of this type is
\begin{equation}\label{eq:kl1d}
   a(y)(x) = 1 + c \sum_{j=1}^\infty j^{-\beta} y_j \sin ( j \pi x ), \quad x \in D=(0,1),\; y \in Y,
\end{equation}
with $\beta > 1$ and sufficiently small $c>0$, corresponding to the choice $\theta_0 = 1$, $\mathcal M = \N$, and $\theta_j(x) = c j^{-\beta} \sin ( j \pi x )$ for each $ j \in\mathcal M$. In the limiting case $\beta =1$, the functions $\theta_j$ result from the Karhunen-Lo\`eve (KL) decomposition of the Brownian bridge on $(0,1)$, and with the random coefficients $y \in Y$ distributed according to $\sigma$ in~\eqref{eq:kl1d} one obtains analogous smoothness properties of random draws of $a(y)$.

However, coefficients $a$ with very similar features can also be obtained by different expansions with multilevel structure. A basic example, again on $D=(0,1)$, are hierarchical piecewise linear hat functions: let $\theta(x) = \max\{  1 - 2\abs{x - \frac12} , 0 \}$ and 
\[
    \theta_{\ell,k} (x) = \theta(2^\ell x - k), \quad (\ell, k) \in \mathcal M = \bigl\{ (\ell, k) \colon \ell \in \N_0, \; k \in \{0,\ldots, 2^\ell - 1\} \bigr\}.
\]
Then for any $\alpha > 0$,
\begin{equation}\label{eq:lc1d}
  a(y)(x) = 1 + c \sum_{(\ell, k)\in\mathcal M}  2^{-\alpha \ell} y_{\ell,k} \theta_{\ell, k}(x), \quad x \in D=(0,1),\; y \in Y,
\end{equation}
yields a random field $a$ with very similar features as~\eqref{eq:kl1d}, but expanded in terms of the locally supported functions $\theta_{\ell,}$ with $\abs{\supp \theta_{\ell,k} } = 2^{-\ell}$. The choice $\alpha= \frac12$ corresponds to the well-known L\'evy-Ciesielski representation of the Brownian bridge \cite{Ciesielski:61}. 
Expansions with similar multilevel structure can also be obtained for more general Gaussian random fields on higher-dimensional domains with Mat\'ern-type or related covariances \cite{BCM:18}.

One advantage of the multilevel representation~\eqref{eq:lc1d} in view of the uniform ellipticity condition~\eqref{eq:uniformellipticity} is that coefficients with lower smoothness than in~\eqref{eq:kl1d} can be realized: for any $\alpha>0$, the series in~\eqref{eq:lc1d} converges absolutely in $L_\infty(0,1)$ uniformly in $y \in Y$, and thus one can parameterize coefficients with arbitrarily low H\"older regularity.

However, representations with multilevel structure as in~\eqref{eq:lc1d} are also advantageous for the convergence rates of adaptive sparse approximations of $u$. 
One observes that the localization in the hat functions $\theta_{\ell,k}$ in~\eqref{eq:lc1d} translates to highly localized features in the Legendre coefficients $u_\nu$, depending on the coordinates that are activated in $\nu$. This is illustrated in the case $\alpha=1$ in Figure \ref{fig:unu}.
As one can recognize, these coefficients have efficient approximations by piecewise linear finite elements on \emph{adaptive} grids. However, these grids clearly need to be chosen with a different local refinement for each $\nu$, so that the corresponding subspaces $V_\nu$ in~\eqref{eq:subspacecV} differ accordingly.

\begin{figure}
\centering
\includegraphics[width=14cm]{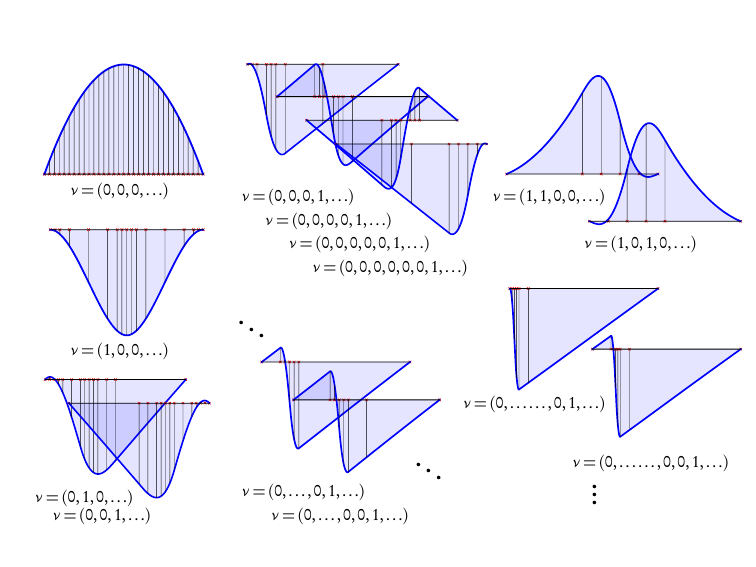}
\vspace{-.8cm}
\caption{Plots of Legendre coefficients $u_\nu \in H^1_0(0,1)$ (normalized to equal $\norm{u_\nu}_{L\infty}$) in the expansion~\eqref{eq:legendreexpansion} with $a$ given by the hierarchical hat function expansion~\eqref{eq:lc1d} with $\alpha=1$, and nodes of adaptively generated piecewise linear approximations on dyadic subintervals.}\label{fig:unu}
\end{figure}

Similar results can be obtained for more general domains $D$ in dimensions $d>1$ when the diffusion coefficient $a$ is given in terms of an expansion with an analogous multilevel structure. The precise assumptions that we use on such multilevel structure for general $d$ are as follows.

\begin{assumption}\label{ass:wavelettheta}
  We assume $\theta_\mu \in L_\infty(D)$ for $\mu \in \cM_0$ such that  the following hold for all $\mu \in \cM$:
  \begin{enumerate}[{\rm(i)}] 
  \item $\diam \supp \theta_\mu\sim 2^{-|\mu|}$, 
  \item $\#\{ \mu \in \cM \colon \abs{\mu } = \ell \} \lesssim 2^{d \ell}$ for each $\ell \in \N_0$ and there exists $M>0$ such that for each $\mu$,
  \[
    \#\{ \mu' \in \cM \colon |\mu|=|\mu'|,\, \supp \theta_{\mu}\cap \supp \theta_{\mu'} \neq \emptyset \} \leq M ,
  \] 
  \item for some $\alpha>0$, one has $\norm{\theta_\mu}_{L_\infty(D)} \lesssim 2^{-\alpha |\mu|}$.
  \end{enumerate}
\end{assumption}

As in the above one-dimensional example, representing $a$ in terms of such multilevel basis functions with localized supports leads to favorable sparse approximability of $u$.
As a consequence of \cite[Cor.~8.8]{BCDS:17}, for sufficiently regular $f$ and $D$ and expansions of $a$ according to Assumption \ref{ass:wavelettheta}, for $\alpha \in (0,1]$ and with $\tilde\cV$ as in~\eqref{eq:subspacecV} we have
\begin{equation}\label{eq:approxrate}
 \min_{v \in \tilde\cV} \norm{u - v}_{\cV} 
   \leq C \Bigl( \sum_{\nu \in F} \dim V_\nu \Bigr)^{-s} \quad \text{for any $\displaystyle s < \begin{cases} \textstyle\frac{\alpha}{d} , & d \geq 2, \\[6pt] \textstyle\frac{2}{3}\alpha , & d = 1 . \end{cases}$}
\end{equation}
The numerical tests in \cite{BV} indicate that with spatial basis functions of sufficiently high approximation order, this also holds for $\alpha>1$.

Remarkably, the limiting convergence rate $\alpha/d$ in~\eqref{eq:approxrate} for $d\geq 2$ is the same as for approximating $u(y)$ by finite elements in $V$ for randomly chosen $y\sim \sigma$ (see \cite{BCD:17}), and also the same as the rate with respect to $\#F$ for semidiscrete approximation of $u$ from $V \otimes \Span \{ L_\nu \}_{\nu \in F}$ obtained in \cite{BCM:17}. In other words, when $a$ is given as a multilevel expansion, when using independent adaptive approximations for each Legendre coefficient $u_\nu$, $\nu \in F$, one can achieve sufficiently strong sparsification that combined spatial-parametric approximations converge at the same rate as only spatial or only parametric approximations. In contrast to results obtained using expansions of $a$ without further structure, these results hold even when convergence is limited by the decay rate $\alpha$ in the coefficient expansion, rather than by the spatial approximation rate of the finite elements.

\subsection{Stochastic Galerkin discretization}

We define $B\colon \cV \to \cV'$ and $\Phi \in \cV'$ by
\begin{equation}\label{eq:Bdef}
    \langle B v , w \rangle =
       \int_Y \int_D \biggl( \theta_0 + \sum_{\mu\in \cM} y_\mu \theta_\mu  \biggr) \nabla v(y) \cdot \nabla w(y)\sdd x \sdd\sigma(y), \quad
    \langle \Phi, w\rangle = \int_Y f\bigl(w(y) \bigr) \sdd\sigma(y) \,.
\end{equation}
The operator $B$ induces an inner product $\langle \cdot, \cdot \rangle_B = \langle B\cdot, \cdot\rangle$ with corresponding norm $\norm{\cdot}_B$ on $\cV$.
Note that
\begin{equation}
  \label{eq:norm equivalence}
  \sqrt{c_B} \norm{ v }_{\cV}= \sqrt{ \inf_{y\in Y} a(y) } \,\norm{ v }_{\cV} \leq \norm{ v }_B \leq  \sqrt{ \sup_{y\in Y} a(y)}\, \norm{ v }_{\cV} = \sqrt{C_B} \norm{ v }_{\cV}.
\end{equation}
As a consequence,
\begin{equation}\label{eq:dual norm equivalence}
  \sqrt{c_B} \norm{v}_B \leq \frac{\langle Bv, v\rangle}{\norm{v}_{\cV}} \leq \norm{ B v }_{\cV'} = \sup_{w \in \cV} \frac{\langle Bv, w\rangle}{\norm{w}_{\cV}} \leq \sup_{w \in \cV} \frac{\langle Bv, w\rangle}{C_B^{-1/2}\norm{w}_{B}} = \sqrt{C_B} \norm{v}_B.
\end{equation}
The stochastic variational formulation of~\eqref{eq:diffusionequation} with coefficient $a$ given by~\eqref{eq:affinecoeff} then reads: find $u \in \cV$ such that
\begin{equation}\label{eq:stochvarform}
   B u = \Phi \quad\text{in $\cV'$.}
\end{equation}
Inserting product Legendre expansions as in~\eqref{eq:legendreexpansion} of $u,v$ into~\eqref{eq:stochvarform} leads to the \emph{semidiscrete form} of the stochastic Galerkin problem for the coefficient functions $u_\nu$, $\nu \in\cF$,
\begin{equation}\label{semidiscreteform}
 \sum_{\mu \in \cM_0} \sum_{\nu' \in \cF} (\bM_\mu)_{\nu,\nu'} A_\mu u_{\nu'}  =  \delta_{0, \nu} f ,
  \; \nu \in \cF,
\end{equation}
where $A_\mu\colon V\to V'$ are defined, for $\mu \in \cM_0$, by
\[
  \langle A_\mu v,w\rangle = \int_D \theta_\mu \nabla v \cdot \nabla w\sdd x\quad \text{for all $v,w\in V$,}
\]
and the mappings $\bM_\mu \colon \ell_2(\cF)\to\ell_2(\cF)$ are given by
\[
\begin{aligned}
 \bM_0  = \left( \int_{Y}  L_\nu(y) L_{\nu'}(y) \sdd\sigma(y)  \right)_{\!\!\nu,\nu'\in\cF}, \qquad
 \bM_\mu = \left( \int_{Y} y_\mu  L_\nu(y) L_{\nu'}(y) \sdd\sigma(y)  \right)_{\!\!\nu,\nu'\in\cF},\; \mu \in \cM.
 \end{aligned}
\]
Since the $L_2([-1,1],\sigma_1)$-orthonormal Legendre polynomials $\{ L_k\}_{k\in \N}$ satisfy the three-term recurrence relation
\[
    y L_k(y) =\sqrt{\beta_{k+1}} L_{k+1}(y) +  \sqrt{\beta_k} L_{k-1}(y), \quad\beta_k = (4 - k^{-2})^{-1} ,
\]
with $L_0 = 1$, $L_{-1} = 0$, $\beta_0 = 0$, we have
\[
\begin{aligned}
   \bM_0 = \bigl( \delta_{\nu,\nu'} \bigr)_{\!\nu,\nu' \in \cF}, \qquad
   \bM_\mu = \biggl( \sqrt{\beta_{\nu_\mu+1}} \,\delta_{\nu+e_\mu, \nu'} + \sqrt{\beta_{\nu_\mu}}\, \delta_{\nu-e_\mu, \nu'} \biggr)_{\!\!\nu,\nu' \in \cF},\; \mu \in \cM,
\end{aligned}
\]
with the Kronecker vectors $e_\mu = ( \delta_{\mu,\mu'} )_{\mu' \in \cM}$.

In the remainder of this work, for simplicity we formulate our method and its analysis for $d=2$. The results carry over immediately to the cases $d=1$ and $d > 2$.
We assume a fixed conforming simplicial coarsest triangulation $\initmesh$ of $D$. If a second (not necessarily conforming) triangulation $\cT$ can be generated from $\initmesh$ by steps of newest vertex bisection, we write $\cT \geq \initmesh$. For such triangulations that are in addition conforming, we consider the standard Lagrange finite element spaces 
\[
  V(\cT) = \PP_1(\cT) \cap V,
\]
where $\PP_1(\cT)$ denotes the functions on $D$ that are piecewise affine on each element of $\cT$.
   
Assume a family of triangulations $\mathbb{T} = (\cT_\nu)_{\nu \in F}$ with finite $F\subset \cF$ and conforming $\cT_\nu \geq \hat\cT_0$ for each $\nu \in F$.
We consider stochastic Galerkin discretization subspaces $\cV(\mathbb{T})$ given in terms of~\eqref{eq:subspacecV} by 
\begin{equation}\label{eq:discrspace}
  \cV(\mathbb{T}) = \biggl\{ \sum_{\nu \in F} v_\nu L_\nu \colon  v_\nu \in V(\cT_\nu),\ \nu \in F  \biggr\} \subset \cV.
\end{equation}
 The total number of degrees of freedom for representing each element of $\cV(\mathbb{T})$ is then $\dim \cV(\mathbb{T})=\sum_{\nu \in F} \dim V(\cT_\nu)$. We use the abbreviation
\begin{equation}\label{eq:Ndef}
  N(\mathbb{T}) = \sum_{\nu \in F} \# \cT_\nu ,
\end{equation}
so that $N(\mathbb{T}) \eqsim \dim \cV(\mathbb{T})$. 

For the finite-dimensional subspaces $\cV(\mathbb{T})\subset \cV$, we consider the stochastic Galerkin variational formulation for $u_\mathbb{T} \in \cV(\mathbb{T})$,
\begin{equation}\label{eq:stochgalerkin}
   \langle B u_\mathbb{T} , v \rangle = \Phi(v) \quad\text{for all $v \in \cV(\mathbb{T})$.}
\end{equation}
As a consequence of~\eqref{eq:uniformellipticity}, the bilinear form given by the left hand side of~\eqref{eq:stochgalerkin} is elliptic and bounded on $\cV$, and by C\'ea's lemma we obtain
\[
   \norm{ u_\mathbb{T} - u }_\cV \leq \frac{2 \norm{\theta_0}_{L_\infty} - c_B}{c_B} \min_{v \in \cV(\mathbb{T})} \norm{v - u}_\cV \,.
\]

\subsection{Adaptive scheme and novel contributions}
The objective of this paper is to construct an adaptive stochastic Galerkin finite element method that can make full use of the approximability result~\eqref{eq:approxrate} by performing independent refinements of meshes for each Legendre coefficient; in other words, for each $u_\nu$ we use a potentially completely different finite element subspace $V(\cT_\nu)$.
Our discretization refinement indicators are obtained from residuals in a stochastic Galerkin formulation of the problem.
The basic strategy for doing so follows our earlier results in \cite{BV}, where we used adaptive wavelet discretizations rather than finite elements for the spatial variable. 
As there, the first step for obtaining residual approximations in the present work is a semidiscrete adaptive operator compression only in the parametric variables, which identifies the relevant interactions between different Legendre coefficients.

Based on this information, the errors in the individual spatial discretizations of the Legendre coefficients need to be handled.
As an initial step towards results on optimal computational costs, as achieved with spatial discretizations by wavelets in \cite{BV}, in the present work we propose and analyze an adaptive method using standard $\PP_1$ finite elements that in each refinement step ensures a fixed error reduction factor, a feature that is also known as \emph{saturation property}. 
This guarantee is in contrast to previous works using independent finite element subspaces $V(\cT_\nu)$ with potentially different meshes $\cT_\nu$ for the different coefficients $u_\nu$, such as \cite{EGSZ:14,CPB:19,BPR21}: in particular, the convergence analysis given in \cite{CPB:19} and \cite{BPR21} is based on \emph{assuming} the saturation property to hold for a certain hierarchical error estimation strategy. 

The technique that we use here to obtain a proof of the saturation property does \emph{not} depend on the wavelet-type multilevel structure of $\{ \theta_\mu \}_{\mu \in \cM}$ as given by Assumption \ref{ass:wavelettheta}. It can be used in the same form whenever each $\theta_\mu$ can be approximated by piecewise polynomials of fixed degree on a sufficiently fine uniform mesh. It hence can similarly be applied to Karhunen-Lo\'eve-type expansions as in \eqref{eq:kl1d}. Thus, our approach would also lead to a provably convergent method for such more general expansions. The main reason why we focus on expansions of the diffusion coefficient with multilevel structure in what follows is that by exploiting the localization of the $\theta_\mu$, we can construct the adaptive solver such that each subroutine has near-optimal computational costs.

One can distinguish two separate difficulties in proving the saturation property for refinements based on the residual in a stochastic Galerkin method. First, in the stochastic variables, we need to identify product polynomial indices that should be added to the approximation subset $F\subset \cF$. Strategies that ensure an error reduction have previously been considered, for example, in \cite{Gittelson:13,EGSZ:15}; these, however, rely on summability of the norms $\norm{\theta_\mu}_{L_\infty}$, which for multilevel expansions of random fields is too strong a requirement.
To exploit the multilevel structure in the expansion~\eqref{eq:affinecoeff} and obtain a scheme that enables linear scaling of computational costs, we instead adapt the approach of \cite{BV} based on operator compression in the stochastic variables. This amounts to adaptively dropping summands in the semidiscrete operator representation $\bigl(\sum_{\mu \in \cM_0} \sum_{\nu' \in \cF} (\bM_\mu)_{\nu,\nu'} A_\mu\bigr)_{\!\nu,\nu'\in\cF}$ as in~\eqref{semidiscreteform} that acts on sequences of Legendre coefficients.

The second and more subtle difficulty arises in the refinement of spatial discretizations. In the case of discretizations with a single spatial mesh (that is, where the meshes $\cT_\nu$ are identical for all $\nu \in F$), standard residual error estimation techniques can be applied as in \cite{EGSZ:15}. Since we use \emph{independent} meshes for each $\nu$, each spatial component of the residual contains sums of jump functionals corresponding to different meshes, which cannot be treated using standard techniques based on Galerkin orthogonality.

Moreover, with such residual components that are proper functionals in $H^{-1}(D)$ arising from integrals over the edges (or facets) of different meshes, it is not clear \emph{a priori} whether a method using a solve-mark-refine cycle can actually achieve a reduction of the error (or a suitable notion of a \emph{quasi-error} \cite{CKNS:08}) by a fixed factor in every iteration: as considered in detail in \cite{CDN:12}, this may require several refinement steps. In \cite{CDN:12}, which addresses adaptive finite elements for standard scalar elliptic problems with right-hand sides that are not in $L_2(D)$, an additional inner refinement loop is introduced. However, this approach is not easily generalizable to our setting, since in the present case the functionals in question depend on the approximate solution.

We thus use a different strategy for error estimation and mesh refinement that is based on residual indicators obtained with a BPX finite element frame \cite{BPX,HS:16,O:94}. Such a frame can be obtained by taking the union over all standard finite element basis functions associated to a hierarchy of grid refinements. Frames of this type have been used in \cite{HSS:08} in the construction of sparse tensor products and in \cite{HS:16} in the analysis of hierarchical error estimation. We approximate $H^{-1}$-norms via frame coefficients similarly as outlined in \cite[Rem.~6.4]{HS:16}, where in contrast to the graded quadrilateral meshes assumed in \cite{HS:16}, we work with meshes refined via standard newest vertex bisection.

\subsection{Outline}
We begin by recapitulating basic properties of newest vertex bisection and of finite element frames in \Cref{sec:frames}. In \Cref{sec:adaptive solver}, we describe the adaptive stochastic Galerkin solver and collect the relevant properties of its constituent procedures, which we use in \Cref{sec:convergence} to prove convergence of the method with a fixed error reduction in each step. 
\Cref{sec:experiments} is devoted to numerical tests that, in view of the expected convergence rates \eqref{eq:approxrate}, hint at  optimality properties of the method. In \Cref{sec:conclusion}, we summarize our findings and point out several directions for further work.

\section{Finite element frames}\label{sec:frames}

Our adaptive scheme is driven by error estimates derived from the norm in $\cV' = L_2(Y, V', \sigma)$ of the residual for a given approximation $v = \sum_{\nu \in \cF} v_\nu L_\nu \in L_2(Y, V, \sigma)$ in the stochastic variational formulation~\eqref{eq:stochvarform},
\begin{equation}\label{eq:res}
   \norm{B v - \Phi}_{\cV'} = \biggl( \sum_{\nu \in \cF} \,\Bignorm{\sum_{\mu \in \cM_0} \sum_{\nu' \in \cF} (\bM_\mu)_{\nu,\nu'} A_\mu v_{\nu'}  -  \delta_{0, \nu} f }_{V'}^2 \biggr)^{\frac12}.
\end{equation}
Since we are mainly interested in the effect of the parameter-dependent diffusion coefficient, to avoid technicalities, in what follows we assume the right-hand side $f$ to be piecewise polynomial on the initial spatial mesh $\hat\cT_0$, from which all meshes arising in the spaces~\eqref{eq:discrspace} are generated by newest vertex bisection.
The treatment of more general $f$ is discussed in Remark~\ref{rem:generalf}. 

Our main task in the approximation of residuals is thus to identify the most relevant indices in the summation over $\nu$ in~\eqref{eq:res}, and for each such $\nu$ to approximate the $V'$-norm of linear combinations of terms of the form $A_\mu v_{\nu'}$, where $v_{\nu'} \in V(\cT_{\nu'})$ for some conforming $\cT_{\nu'} \geq \initmesh$. For the latter task, we use finite element frames derived from a hierarchy of uniform refinements of~$\initmesh$.

For newest vertex bisection, our starting point is a general initial triangulation $\initmesh$, which we endow with an edge labelling that is admissible in the sense of \cite[Lem.~2.1]{BDD}; that is, each edge in $\initmesh$ is labelled either~$0$ or~$1$ such that each triangle in $\initmesh$ has exactly two edges labelled~$1$ and one edge labelled~$0$.
Such a labelling for general conforming two-dimensional triangulations can always be found~\cite[Lem.~2.1]{BDD}, but requires solving a combinatorial problem. For larger initial meshes, it suffices instead to bisect every triangle twice and take the resulting triangulation as $\initmesh$, where edges that connect two newly created vertices are labelled by $0$ and all others by $1$.

Newest vertex bisection is then applied to a triangle with labels $(i,i,i-1)$ by bisecting the edge with the lowest label $i-1$ and assigning the label $i+1$ to both halves of the bisected edge and to the newly added bisecting edge, so that the two newly created triangles both have labels $(i+1,i+1,i)$ and the edges opposite the newly added vertex will be the next to be bisected. The meshes generated by newest vertex bisection are uniform shape regular, dependent only on the initial triangulation~$\hat\cT_0$.

\begin{remark}

Bisection can be used in an analogous manner in higher dimensions, as shown by Stevenson~\cite{Stevenson:NVB08}. The essential condition is a guarantee that uniform refinements are again conforming. A construction that guarantees this condition can again be obtained by subdivision; see the appendix of~\cite{Stevenson:NVB08}. Although we focus here on $d=2$, using these features of the higher-dimensional analogue of newest vertex bisection, the developments that follow can be extended to $d\geq 3$. 
\end{remark}

Note that after two applications of newest vertex bisection to a mesh with admissible initial labelling, every edge has been bisected once. This gives rise to a hierarchy of meshes $\hat\cT_1, \hat\cT_2,\ldots$, where $\hat\cT_{j+1}$ is obtained from $\hat\cT_j$ by applying two passes of newest vertex bisection to the full mesh, such that a new node is added to each edge in $\hat\cT_j$.
For each $j$, we define $\psi_\lambda$ with $\lambda = (j,k)$ as the enumeration of piecewise affine hat functions in $V(\hat \cT_j)$, normalized such that $\norm{ \psi_\lambda}_V = 1$, where $k = 1,\ldots, N_j = \dim V(\hat\cT_j)$.
We assume that $\hat \cT_0$ gives rise to a nontrivial finite element space, i.e., $N_0 > 0$, and set
\[
\Theta =  \bigl\{ (j,k) \colon  j = 0,1,2,\ldots , \; k = 1,\ldots, N_j  \bigr\}.
\]
The family $\psi_\lambda$, $\lambda\in \Theta$, of hat functions on all levels of the uniformly refined grid hierarchy (that is, the function system underlying the classical BPX preconditioner \cite{BPX}) is then a \emph{frame} of $V$, which means that for $\xi \in V'$, we have a proportionality with uniform constants between $\norm{(\langle \xi,\psi_\lambda\rangle)_{\lambda \in \Theta}}_{\ell_2(\Theta)}$ and $\norm{\xi}_{V'}$.
Under the given assumptions on the mesh hierarchy $(\hat\cT_j)_{j\geq 0}$, we obtain the result of \cite[Thm.~5]{HSS:08} in the following form.
 
\begin{lemma}\label{lm:frame}
   The family $\Psi = ( \psi_\lambda)_{\lambda \in \Theta}$ is a frame of $V$, that is, there exist $c_\Psi, C_\Psi >0$ depending only on $\initmesh$ such that
   \begin{equation}\label{eq:framecondition}
     c_\Psi^2 \norm{ \xi }_{V'}^2 
     \leq \sum_{\lambda \in  \Theta} \abs{\langle \xi, \psi_\lambda\rangle }^2 
     \leq C_\Psi^2 \norm{ \xi }_{V'}^2,\; \xi \in V'.  
   \end{equation}
\end{lemma}

We assume each function $\theta_\mu$, $\mu \in \cM$, to be piecewise polynomial on a triangulation $\hat\cT_j$ for some $j$ depending on $\abs{\mu}$ in the following sense.

\begin{assumption}\label{ass:pwpolytheta}
  There exists $m\in\N_0$, $k\in \N_0$ and $K\in \N$ such that for all $\mu \in \cM$, $\theta_\mu \in \PP_m(\hat\cT_{\abs{\mu}+k})$ and $\#\{T\in \hat\cT_{\abs{\mu}+k}\colon \supp \theta_\mu \cap T\neq \emptyset\}\leq K$.
\end{assumption}

\begin{remark}\label{rem:MultilevelStructure}
Concerning the role of these assumptions, the following comments are in order:
\begin{enumerate}[{\rm(a)}]
\item The case of non-polynomial functions $\theta_\nu$ can be treated by replacing them by polynomial approximations. 
Approximability of the $\theta_\mu$ of a similar kind are required in any case for quadrature.
\item In its stated form, Assumption \ref{ass:pwpolytheta} applies to typical expansions with multilevel structure such as wavelet-like systems. For example, for functions in wavelet-type expansions of Mat\'ern random fields constructed in \cite{BCM:18}, it is not difficult to obtain suitable efficient polynomial approximations with uniform accuracy that satisfy Assumption \ref{ass:pwpolytheta}, see \cite[Eq.~(139)]{BCM:18}
\item\label{rem:MultiLevelErrorReduction} The multilevel structure of $\theta_\mu$ implied by Assumption \ref{ass:pwpolytheta} is not required for our main result of uniform error reduction in each step of our adaptive scheme, and thus in particular not for the convergence of the method. As a requirement to obtain this, it would suffice to assume that for a fixed $m$, for each $\mu \in \cM$ there exists $j(\mu)$ such that $\theta_\mu \in \PP_m(\hat\cT_{j(\mu)})$. However, the stronger Assumption \ref{ass:pwpolytheta} is crucial for the favorable computational complexity of the adaptive method. We return to this point in \Cref{rem:KLgeneralization}.
\item The same comments apply to $d\geq 3$, where Assumption \ref{ass:pwpolytheta} in its basic form means that each $\theta_\mu$ is represented (or approximated) by a piecewise polynomial on a mesh in a hierarchy of uniform refinements and has wavelet-type level-dependent localization.
\end{enumerate}
\end{remark}

For approximating dual norms in $V'$ of sums of functionals of the form $A_\mu v_\nu$ with $v_\nu \in V(\cT_\nu)$, with $\cT$ the joint refinement of $\cT_\nu$ and the mesh on which $\theta_\mu$ is piecewise polynomial, we use integration by parts to obtain
\begin{equation}\label{eq:intbyparts}
  \langle A_\mu v_\nu , w\rangle 
   = \int_D \theta_\mu \nabla v_\nu \cdot\nabla w\sdd x 
    = \sum_{K \in \cT} \left\{ 
      \int_{\partial K \cap D} \theta_\mu n \cdot  \nabla v_\nu \, w\sdd s 
      - \int_K \nabla \cdot (\theta_\mu \nabla v_\nu)\, w\sdd x\right\}
\end{equation}
for $w \in V$. Here the integrands $\theta_\mu n \cdot  \nabla v_\nu$ and $\nabla \cdot (\theta_\mu \nabla v_\nu)$ are polynomials on the respective subdomains by \Cref{ass:pwpolytheta}. We can thus apply the following result on the decay of $\ell_2$-tails of frame coefficients.
A related estimate is obtained in \cite[Sec.~6]{HS:16} with frames on graded quadrilateral meshes.

We define $\level(E)$ for an edge~$E$ and $\level(T)$ for a triangular element~$T$ as  the unique $j$ such that the uniform mesh~$\hat\cT_j$ contains $E$ or either $T$ or a bisection of $T$, respectively. Similarly, for the enumeration indices~$\lambda = (j,k)\in \Theta$, we define $\abs{\lambda} = j$.

\begin{lemma}\label{lmm:spatialrelerr}
   Let $d_1, d_2 \in \N_0$.
   There exist $C>0$ and $J_0 \in \N$ depending only on $\initmesh$ and $d_1, d_2$ such that for all $J\geq J_0$ the following holds:
 For any $\cT \geq \initmesh$ with interior edges $\mathcal{E}$ and any $\xi \in V'$ of the form 
 \begin{equation}\label{eq:xirepresentation}
      \langle \xi, v\rangle = \sum_{K \in \mathcal{T}} \int_K \xi_{K} v \sdd x + \sum_{E \in \mathcal{E}}  \int_E \xi_E v \sdd s ,\quad v \in V,
 \end{equation}
 where $\xi_K \in \PP_{d_1}(K)$, $K \in \mathcal{T}$, and $\xi_E \in \PP_{d_2}(E)$, $E \in \mathcal{E}$, we have
 \[
     \Bigl( \sum_{\lambda \in \Theta_{J}(\cT)} \bigabs{\langle \xi, \psi_\lambda\rangle }^2   \Bigr)^{\frac12} \leq C 2^{-J}  \Bigl(  \sum_{\lambda \in \Theta}  \bigabs{\langle \xi, \psi_\lambda\rangle }^2 \Bigr)^{\frac12} 
 \]
 with
 \begin{multline}\label{eq:ThetaJdef}
   \Theta_{J}(\cT) = \bigl\{  \lambda \in\Theta \colon  (\forall K \in \mathcal{T}\colon    \meas_2(\supp \psi_\lambda \cap K)> 0 \implies \abs{\lambda} > \level(K) + J )   \\
   \wedge   (\forall E \in \mathcal{E}\colon    \meas_1(\supp \psi_\lambda \cap E)>0 \implies \abs{\lambda} > \level(E) + 2J ) 
   \bigr\} .
 \end{multline}
\end{lemma}

\begin{proof}
  For $\lambda \in \Theta_{J}(\cT)$ we have, by uniform shape regularity and the definition of $\Theta_{J}(\cT)$, that $\psi_\lambda$ have support on a uniformly bounded number of elements $K \in \mathcal{T}$. Thus
\begin{equation}\label{eq:squareest}
    \Bigabs{  \sum_{K \in \mathcal{T}}  \bigabs{\langle \xi_K, \psi_\lambda \rangle}   }^2
     \lesssim \sum_{K \in \mathcal{T}} \bigabs{\langle \xi_K, \psi_\lambda \rangle}^2,
     \qquad
     \Bigabs{  \sum_{E \in \mathcal{E}}  \bigabs{\langle \xi_E, \psi_\lambda \rangle}   }^2
     \lesssim \sum_{E \in \mathcal{E}} \bigabs{\langle \xi_E, \psi_\lambda \rangle}^2
\end{equation}
with uniform constants.
Let
\[
   h_K = 2^{-\level(K)}, \quad h_E = 2^{-\level(E)}   \,.
\]
By~\eqref{eq:squareest} and using that by the normalization of $\psi_\lambda$, with $j = \abs{\lambda}$ we have that $\norm{\psi_\lambda}_{L^2(K)} \lesssim 2^{-j}$ and $\norm{\psi_\lambda}_{L^2(E)}\lesssim 2^{-\frac12 j}$.
Thus
   \begin{align*}
   \sum_{\lambda \in \Theta_{J}(\cT)} \bigabs{\langle \xi, \psi_\lambda\rangle }^2
     & \lesssim \sum_{K \in \mathcal{T}} \sum_{\lambda \in \Theta_{J}(\cT)} \bigabs{\langle \xi_K, \psi_\lambda \rangle}^2 +  \sum_{E \in \mathcal{E}} \sum_{\lambda \in \Theta_{J}(\cT)} \bigabs{\langle \xi_E, \psi_\lambda \rangle}^2    \\
     & \lesssim  \sum_{K \in \mathcal{T}} \sum_{j > \level(K) + J} 2^{-2j} \norm{\xi_K}_{L_2(K)}^2
        + \sum_{E \in \mathcal{E}} \sum_{j > \level(E) + 2J} 2^{- j} \norm{\xi_E}_{L_2(E)}^2 \\
   & \lesssim 2^{-2J} \sum_{K \in \mathcal{T}} h_K^2 \norm{\xi_K}_{L_2(K)}^2 
     + 2^{-2J} \sum_{E \in \mathcal{E}} h_E \norm{\xi_E}_{L_2(E)}^2 \,.
\end{align*}

Since for all $K \in \cT$ and $E \in \mathcal{E}$, the components $\xi_K$ and $\xi_E$ are polynomial with degrees bounded by $d_1$ and $d_2$, respectively, \cite[Thm.~3.59]{V:13} yields
\[
  \sum_{K \in \mathcal{T}} h_K^2 \norm{\xi_K}_{L_2(K)}^2 
  + \sum_{E \in \mathcal{E}} h_E \norm{\xi_E}_{L_2(E)}^2 \lesssim \norm{\xi}_{V'}^2
\]
with a constant depending only on the initial triangulation $\hat\cT_0$ and on $\max\{d_1,d_2\}$.
By~\eqref{eq:framecondition},
\[
   \norm{\xi}_{V'}^2
      \lesssim \sum_{\lambda \in \Theta}  \bigabs{\langle \xi, \psi_\lambda\rangle }^2
\]
with a further constant depending only on $\hat\cT_0$.
This concludes the proof.
\end{proof}

We use Lemma \ref{lmm:spatialrelerr} to estimate the $V'$-norms in \eqref{eq:res}: we first determine the relevant indices $\nu \in \cF$ in the summation on the right-hand side of \eqref{eq:res}. This is done by an operator compression in the parametric degrees of freedom that is independent of the spatial meshes. For each of these $\nu$, we evaluate the frame coefficients
\[
\br_\nu = (\br_{\nu,\lambda})_{\lambda \in \Theta_\nu}, \quad \br_{\nu,\lambda} = \biggl\langle \sum_{\nu' \in \cF} \sum_{\mu \in \cM_0(\nu,\nu')}  (\bM_\mu)_{\nu,\nu'} A_\mu v_{\nu'}, \psi_\lambda\biggr\rangle
\]
for suitable finite subsets $\cM_0(\nu,\nu')\subset \cM_0$ and for $\lambda \in \Theta_\nu \subset \Theta$. Here $\Theta_\nu$ is chosen according to Lemma \ref{lmm:spatialrelerr} such that 
\[
   \sum_{\lambda \notin \Theta_\nu} \br_{\nu,\lambda}^2 \leq \zeta^2  \sum_{\lambda\in \Theta}  \br_{\nu,\lambda}^2  
\]
with a sufficiently small relative error $\zeta \in (0,1)$.

The indices $(\nu,\lambda)$ corresponding to the largest values $\abs{\br_{\nu,\lambda}}$ are selected for refinement, based on a condition analogous to the D\"orfler criterion, by a tree thresholding procedure. Subsequently, for each $\nu$, the associated current mesh $\cT_\nu$ is refined such that all selected frame elements are contained in the resulting finite element space. 
A similar strategy has been outlined for the refinement of a single finite element mesh in \cite{HS:16}.
It is especially in our setting, with interactions between many different meshes in elliptic systems of PDEs, that this technique provides crucial flexibility compared to standard approaches of a posteriori error estimation.

\section{Adaptive solver}\label{sec:adaptive solver}

In this section, we describe the adaptive solver and analyze its constituent procedures. As noted above, the adaptive scheme is based on Galerkin discretizations on successively refined meshes. Here, the way in which the underlying residual error indicators are obtained and used in the mesh refinement differs from previous approaches to adaptive stochastic Galerkin finite element methods. We combine adaptive operator compression in the stochastic degrees of freedom with frame-based spatial refinement indicators. These are subsequently used in a tree-based selection of refinements that in each iteration of the adaptive scheme permits the application of multiple refinement steps within single mesh elements; this latter property is crucial in ensuring error reduction by a uniform factor.

The precise procedures for the individual steps in the adaptive scheme are described in this section. They are used by the adaptive solver given in \Cref{alg: adaptive method}, which is structured as follows:

\begin{enumerate}[{\bf(i)}]
\item Finite element frame coefficients of the current residual are approximated by the procedure \textsc{ResEstimate} (\Cref{alg: ResEstimate}), which in turn has two main steps:
\begin{enumerate}[(a)]
\item  First, the semidiscrete operator compression \textsc{Apply} given in \Cref{apply_semidiscr}, which selects spatial operator components in the elliptic system of PDEs \eqref{semidiscreteform} resulting from the stochastic Galerkin formulation. This step can be done analogously to the construction given in \cite{BV}.
\item  Second, finite element frame coefficients of residuals are assembled in \Cref{alg: ResEstimate} with the aid of \textsc{DualNormIndicators} (\Cref{dualnormindicators}) and \textsc{Sum} (\Cref{sumalg}). These schemes use localization of the residual contributions according to \Cref{lmm:spatialrelerr}.
\end{enumerate}
\item In the next step, a selection of frame coefficients is used to produce suitable spatial meshes. This is described in \Cref{sec:refinetriangulation} with the procedures \textsc{TreeApprox} and \textsc{Mesh}, replacing the standard mark-refine step in adaptive finite elements. 
\item A Galerkin solution for the new discretization is computed using \textsc{GalerkinSolve} (\Cref{alg:galerkin solve}), combining semidiscrete operator compression with standard optimal finite element preconditioning for each spatial component, as described in \Cref{sec: galerkin solver}.
\end{enumerate}

\subsection{Residual estimation}\label{sec:residual estimation}

As a first first step in the adaptive scheme, we require a semidiscrete adaptive compression of the operator $B\colon \mathcal V\to\mathcal V'$. For $\ell \in \mathbb N_0$ we define the truncated operators $B_\ell$ by
\begin{equation}\label{eq:Btrunc}
  \langle B_\ell v , w \rangle = \int_Y \int_D \Bigl(\theta_0 + \sum_{\substack{\mu \in \mathcal{M} \\ \abs{\mu} < \ell}} y_\mu\theta_\mu \Bigr) \nabla v(y) \cdot \nabla w(y)\sdd x \sdd\sigma(y) \quad \text{for all $v,w \in \cV$.}
\end{equation} 
We make use of two particular consequences of \Cref{ass:wavelettheta}: First, for the expansion functions $(\theta_\mu)_{\mu\in\cM}$ we have that there exists $C_1>0$ such that for all $\ell \geq 0$,
\begin{equation}\label{eq:multilevel1}
\#\{ \mu : |\mu| = \ell \} \leq C_1 2^{d\ell}.
\end{equation}
Second, there exists $C_2>0$ such that with the $\alpha>0$ in \Cref{ass:wavelettheta}(iii), for all $\ell \geq 0$,
\begin{equation}
\label{eq:multilevel2}
  \sum_{|\mu|=\ell} \abs{\theta_{\mu}} \leq C_2 2^{-\alpha  \ell} \quad \text{a.e.~in $D$.}
\end{equation}
As in \cite[Prop.~3.2]{BV}, we have the following approximation result.

\begin{prop}\label{prop:operator compression}
   With $C_2 >0$ as in~\eqref{eq:multilevel2}, for all $\ell \geq 0$,
  \[
   \norm{B - B_\ell}_{\cV\to\cV'} = \Bignorm{ \sum_{\abs{\mu} \geq \ell}  \bM_\mu \otimes A_\mu  }_{\ell_2(\cF)\otimes V \to \ell_2(\cF)\otimes V'} \leq \frac{C_2}{1 - 2^{-\alpha}} 2^{-\alpha \ell}.
  \]
\end{prop}

We introduce the following notation: for $\nu \in \cF$ and $v \in \cV$,
\[
   [ v ]_\nu =  \int_Y v L_\nu(y)\,\sdd\sigma(y) ,
\]
and similarly, for $\xi \in \cV' = L_2(Y,V',\sigma)$, 
\begin{equation}\label{eq:functionalcoeffs}
   [  \xi ]_\nu = \int_Y \xi L_\nu(y)\,\sdd\sigma(y)   \,,
\end{equation}
so that
\[
   \langle \xi , v \rangle_{\cV',\cV} = \sum_{\nu \in \cF} \langle [\xi]_\nu, [v]_\nu\rangle_{V',V}  .
\]
Note that $\norm{v}_\cV = \norm{(\norm{[v]_\nu}_V)_{\nu \in \cF}  }_{\ell_2}$
and thus also $\norm{\xi}_{\cV'} = \norm{(\norm{[\xi]_\nu}_{V'})_{\nu \in \cF}  }_{\ell_2}$.

For sequences $\bv \in \ell_2(\cF)$ and $s>0$, we define the standard approximation spaces $\cA^s = \cA^s(\cF)$ with quasi-norm 
\begin{equation}
  \norm{\bv}_{\cA^s(\cF)} = \sup_{N \in \N_0} ( 1 + N)^s \min \{ \norm{ \bv - \bw }_{\ell_2(\cF)} \colon \#\supp \bw \leq N \}.
\end{equation}
Thus for $v \in L_2(Y, V, \sigma)$ with $\bignorm{ \bigl( \norm{[v]_\nu}_V \bigr)_{\nu \in \cF}  }_{\cA^s} < \infty$, for each $N \in \N$ there exists $F_N \subset \cF$ with $\# F_N \leq N$ such that
\[
  \biggnorm{ v - \sum_{\nu \in F_N} [v]_\nu L_\nu }_{L_2(Y,V,\sigma)} \leq (N+1)^{-s} \bignorm{ \bigl( \norm{[v]_\nu}_V \bigr)_{\nu \in \cF}  }_{\cA^s}.
\]

Similarly as in \cite{BV}, we construct a routine \textsc{Apply} in \Cref{apply_semidiscr} taking in a tolerance $\eta>0$ and $v\in \cV$ with finite stochastic support $\supp ([v]_\nu)_{\nu \in \cF}< \infty$ that produces a blockwise operator compression, adapted to $v$ and encoded by subsets of $\cM_0\times \supp ([v]_\nu)_{\nu \in \cF}$. These subsets of indices specify which truncation of $B$ as in~\eqref{eq:Btrunc} should be applied to which subset of Legendre coefficients of $v$. The following result is obtained by minor modifications of the proof of \cite[Prop.~4.8]{BV}.

\begin{algorithm}[t]
	\caption{$( M(\nu) )_{\nu \in F}, (F_i, \ell_i)_{i = 0}^I = \text{\textsc{Apply}}(v; \eta)$, for $N := \#\supp ([v]_\nu)_{\nu \in \cF} < \infty$, $\eta>0$.} \label{apply_semidiscr}
\begin{enumerate}[{\bf(i)}]
\item If $\norm{B}_{\cV\to\cV'} \norm{v}_{\cV} \leq \eta$, return the empty tuple with $F = \emptyset$;
otherwise, with $\bar I:=\ceil{\log_2 N}$, for $i = 0,\ldots, \bar I$, determine $F_i \subset \cF$ such that $\# F_i \leq 2^{i}$ and $P_{F_i} v = \sum_{\nu \in F_i} [v]_\nu L_\nu$ satisfies
\begin{equation*}
    \norm{v- P_{F_i} v }_{\cV} \leq C \min_{\# \tilde F \leq 2^i} \norm{ v - P_{\tilde F} v }_\cV  
\end{equation*}
with an absolute constant $C>0$.
Choose $I$ as the minimal integer such that
\[
  \delta = \norm{B}_{\cV\to\cV'} \norm{v - P_{ F_I} v }_\cV \leq \frac{\eta}{2}.
\]

\item With $d_0 =P_{F_0} v$, $d_i = (P_{F_i} - P_{F_{i-1}} )v$, $i=1,\ldots,\bar I$, and $N_i = \# F_i$, set 
\begin{equation*}
	\ell_i = \left\lceil\alpha^{-1} \log_2 \biggl( \frac{C_B}{\eta - \delta} \biggl( \frac{\norm{d_i}_\cV}{N_i} \biggr)^{\frac{\alpha}{\alpha+d}} \Bigl( \sum_{k=0}^I \norm{d_k}_\cV^{\frac{d}{\alpha+d}} N_k^{\frac{\alpha}{\alpha+d}} \Bigr) \biggr)\right\rceil,
	\quad i = 0, \ldots, I.
\end{equation*}

\item With $g$ defined by
  \begin{equation*}
    g = \sum_{i=0}^I  B_{\ell_i} d_i ,
  \end{equation*}
  for each $\nu \in F = \supp ([g]_\nu)_{\nu \in \cF}$, collect the sets $M(\nu) \subset \cM_0 \times \supp ([v]_\nu)_{\nu\in\cF}$ of minimal size such that
\begin{equation*}
 [g]_\nu = \sum_{(\mu,\nu') \in M(\nu)}  (\bM_{\mu})_{\nu,\nu'} \,A_{\mu} \, v_{\nu'},\quad \nu \in F,
\end{equation*}
and return $( M(\nu) )_{\nu \in F}$ as well as $(F_i, \ell_i)_{i = 0}^I$.
\end{enumerate}
\end{algorithm}

\begin{prop}\label{semidiscrapply}
	Let $s >0$ with $s< \frac{\alpha}{d}$, let $B$ be as in~\eqref{eq:Bdef}, let $v$ satisfy $\#\supp ([v]_\nu)_{\nu\in\cF} < \infty$, and let $\ell_i$, $i=0,\ldots,I$, and $g$ be as defined in \Cref{apply_semidiscr}. Then $\norm{B v - g}_{\cV'} \leq  \eta$, for $F = \supp ([g]_\nu)_{\nu \in \cF}$ we have
    \begin{equation}\label{eq:Fest}
    \# F \leq    \sum_{\nu \in F} \# M(\nu)
     \lesssim  \sum_{i=0}^I 2^{d \ell_i} \# F_i   \lesssim \eta^{-\frac1s} \bignorm{\bigl(\norm{[v]_\nu}_V \bigr)_{\nu\in\cF}}_{\cA^s}^{\frac1s},
    \end{equation}
    and 
    \begin{equation}\label{eq:elljest}
      \max_{i=0,\ldots,I} \ell_i \lesssim 1 + \abs{\log \eta} + \log \bignorm{\bigl(\norm{[v]_\nu}_V \bigr)_{\nu\in\cF}}_{\cA^s}.
    \end{equation}
  The constants in the inequalities depend on $C$ as in~\Cref{apply_semidiscr}, $C_B$, $d$, $\alpha$, $s$, and on $C_1$ from~\eqref{eq:multilevel1}.
\end{prop}

As a next step, for $g$ and $F$ as defined in~\Cref{apply_semidiscr}, for each $\nu \in F$ we need to evaluate
\begin{equation}\label{eq:semidiscrwdefM}
  [g]_\nu = \sum_{(\mu,\nu') \in M(\nu)}  (\bM_{\mu})_{\nu,\nu'} \,A_{\mu} \, v_{\nu'} \,\in\, V'.
 \end{equation}
We represent each $[g]_\nu$ as a piecewise polynomial on a (not necessarily conforming) triangulation as in \Cref{lmm:spatialrelerr}. 
To this end, we first compute each summand 
\[
  [g]_{\nu,\nu'} = (\bM_{\mu})_{\nu,\nu'} \,A_{\mu} \, v_{\nu'}
\]
in~\eqref{eq:semidiscrwdefM}.
Note that the index $\mu$ in the definition $[g]_{\nu,\nu'}$ is the unique index such that $\nu' = \nu \pm e_\mu$. 
For fixed $\nu'$, we can compute all $[g]_{\nu,\nu'}$ for $\nu\in F$ by traversing the mesh of $v_{\nu'}$ once. Let $\mathcal K_{\nu'}$ denote the elements in the mesh of $v_{\nu'}$ and let $\ell_{\nu'} = \max_{\nu'\in F_i}\ell_i$. \Cref{ass:wavelettheta}~(ii) guarantees that each element $K\in\mathcal K_{\nu'}$ is required for at most $C \ell_{\nu'} (2^{d\ell_{\nu'}-\level(K)})$ summands, with a uniform constant $C$. Hence, computing all $[g]_{\nu,\nu'}$ for $\nu\in F$ has a complexity of order~$O(\ell_{\nu'}\# \mathcal K_{\nu'} + 2^{d\ell_{\nu'}})$. However, for fixed $\nu$, the summands $[g]_{\nu,\nu'}$ are on different meshes.

To efficiently evaluate the sum, we use the natural tree structure on the triangles that are generated by newest vertex bisection, that is, the children of a triangle are the two triangles generated by bisection. The routine \textsc{Sum} in \Cref{alg: sum} yields a representation of $[g]_\nu$ on a joint grid, with a computational complexity that is linear in the sum of the sizes of the triangulations of $[g]_{\nu,\nu'}$.

\begin{remark}
  The step~\eqref{step:preadding} in \Cref{alg: sum} can be executed while computing $[g]_{\nu,\nu'}$ without explicitly storing $[g]_{\nu,\nu'}$.
\end{remark}
\begin{algorithm}[t]
	\caption{$ \bar\xi = \text{\textsc{Sum}}((\xi_\nu)_{\nu \in F})$ with finite $F\subset \cF$ and $\xi_\nu$ is of the form~\eqref{eq:xirepresentation} on a mesh $\cT_\nu$ with edges $\mathcal{E}_\nu$ for each $\nu \in F$.}\label{alg: sum}
  \begin{flushleft}
Let $\xi_{\nu,K}$ and $\xi_{\nu, E}$ be the polynomials of this representation for each $K \in \cT_\nu$ and $E \in \mathcal{E}_\nu$, respectively.
For each triangle $K$ let $E(K)$ denote the edge that is bisected by newest vertex bisection of $K$, and $K_i$ and $E_i(K)$ for $i = 1,2$ be the corresponding bisected elements and edges.
  \end{flushleft}
\begin{enumerate}[{\bf(i)}]
\item Initialize $\bar\xi$: Let $\bar\xi_K=0$ and $\bar\xi_E=0$ for every element and every edge;

\item \label{step:preadding} For every $K$ and $E$ in the elements and edges of $\mathcal T_\nu$ and every $\nu$, add the respective polynomial to $\bar\xi$:
\[
  \bar\xi_K \leftarrow \bar\xi_K + \xi_{\nu,K};\quad \bar\xi_{E} \leftarrow \bar\xi_E + \xi_{\nu,E}.
\]
\item \label{step:buildtree} Set $\mathfrak{T} = \emptyset$ and for each $K$ such that $\bar\xi_K\neq 0$ or $\bar\xi_{E(K)}\neq 0$, insert all ancestors of $K$ into $\mathfrak{T}$, so that $\mathfrak{T}$ becomes a tree with root elements $\mathfrak{R} \subseteq \hat\cT_0$;

\item While $\mathfrak{T}$ is not empty:

\noindent
\quad For all $K$ in $\mathfrak{R}$ add polynomials to the corresponding children in $\mathfrak{T}$:
  \begin{align*}
    \bar\xi_{K_i} &\leftarrow \bar\xi_{K_i} + \bar\xi_{K} \text{ for $i =1,2$} ;&
    \quad \bar\xi_{K} &\leftarrow 0 ;\\
    \quad \bar\xi_{E_i(K)} &\leftarrow \bar\xi_{E_i(K)} + \bar\xi_{E(K)}\text{ for $i =1,2$};
    & \quad \bar\xi_{E(K)} &\leftarrow 0;
  \end{align*}
  \quad remove $K$ from $\mathfrak{T}$ and $\mathfrak{R}$ and insert $K_1$, $K_2$ into $\mathfrak{R}$;
\item Return $\bar\xi$.
\end{enumerate}
\label{sumalg}
\end{algorithm}

\subsubsection{Approximate dual norm evaluation}
By \Cref{lm:frame}, we have the equivalent expression for the dual norm of the approximate residual
\[
  \norm{w}_{\cV'}^2 = \norm{(\norm{[g]_\nu}_{V'})_{\nu \in \cF}  }^2_{\ell_2}\eqsim
\sum_{\nu\in F} \sum_{\lambda\in \Theta}\abs{\langle [g]_\nu, \psi_\lambda \rangle}^2.
\]
With the help of \Cref{lmm:spatialrelerr} we can estimate the latter expression by evaluating the coefficients $\langle [g]_\nu, \psi_\lambda \rangle$ for all $\lambda\in \Theta\setminus\Theta_J(\cT)$ for some $J$. We now ensure that the costs for computing these coefficients are linear in the size of $\Theta\setminus\Theta_J(\cT)$.
For any frame element~$\psi_\lambda$, we define
$\mathcal K(\psi_\lambda) $ as the supporting triangles, $\mathcal E(\psi_\lambda)$ as the corresponding interior edges, and 
\[
  \mathcal R(\psi_\lambda) = \{\psi_\mu\colon \abs{\mu} = \abs{\lambda} + 1\text{ and } \supp(\psi_\mu)\subset \supp(\psi_\lambda)\}  
\]
as the set of frame elements in terms of which  $\psi_\lambda$ can be represented on the next higher level of refinement, that is, there are coefficients $h_{\lambda, mu}$ such that $\psi_\lambda = \sum_{\psi_{\mu}\in \mathcal R(\psi_\lambda)} h_{\lambda,\mu} \psi_{\mu}$.

\begin{lemma}\label{lm: Triangulation to frame elements}
    Let $J\in\N$ and let $\cT \geq \hat\cT_0$ be a triangulation with edges $\mathcal{E}$.
    Then for $\Theta_J(\cT)$ as in~\eqref{eq:ThetaJdef}, we have $\#(\Theta\setminus\Theta_{J}(\cT)) \lesssim 4^J\#\cT $. 
\end{lemma}
\begin{proof}
    By construction of $\psi_\lambda$, for $K\in \cT$ we have
    \[
    \# \bigl\{  \lambda\in \Theta \colon \meas_2(\supp \psi_\lambda \cap K)> 0 \text{ and } \abs{\lambda} \le \level(K) + J  \bigr\} \lesssim 4^J ,
    \]
    and for any edge $E\in \mathcal{E} $, 
    \[
    \# \bigl\{  \lambda\in \Theta \colon \meas_1(\supp \psi_\lambda \cap E)> 0 \text{ and } \abs{\lambda} \le \level(E) + 2J  \bigr\} \lesssim 4^J.
    \]
    Since
    \begin{multline*}
      \Theta\setminus\Theta_{J}(\cT) = 
      \bigl\{  \lambda \in\Theta \colon  (\exists K \in \cT\colon    \meas_2(\supp \psi_\lambda \cap K)> 0 \wedge \abs{\lambda} \le \level(K) + J )   \\
      \vee   (\exists E \in \mathcal{E}\colon    \meas_1(\supp \psi_\lambda \cap E)>0 \wedge \abs{\lambda} \le \level(E) + 2J ) 
      \bigr\} ,
    \end{multline*}
    we have $\#\bigl(\Theta\setminus\Theta_{J}(\cT)\bigr) \lesssim 4^J\#\cT $.
\end{proof}

\begin{lemma}\label{lm: BPXcoeffscosts}
  Let $J \in\N$. 
  For any not necessarily conforming triangulation $\cT $ with interior edges $\mathcal{E}$ and any $\xi \in V'$ of the form 
\eqref{eq:xirepresentation},
where $\xi_K \in \PP_{d_1}(K)$, $K \in \mathcal{T}$, and $\xi_E \in \PP_{d_2}(E)$, $E \in \mathcal{E}$,
the evaluation of the coefficients
  $\langle \xi, \psi_\lambda\rangle$ for all $\lambda \in \Theta\setminus\Theta_{J}(\cT)$
requires at most $C\#( \Theta\setminus\Theta_{J}(\cT) )$ basic operations, where $C$ depends only on $d_1, d_2$ and on $\max_{\lambda\in\Theta}\#\mathcal K(\psi_\lambda)$, $\max_{\lambda\in\Theta}\#\mathcal E(\psi_\lambda), $ and $ \max_{\lambda\in\Theta}\#\mathcal R(\psi_\lambda) $.

\end{lemma}
  
\begin{proof}
  First, if $\xi$ is polynomial on $K\in \mathcal K(\psi_\lambda)$ and $E\in \mathcal E(\psi_\lambda)$, then $\langle \xi, \psi_\lambda\rangle $ can be evaluated in $c\, (\#\mathcal K(\psi_\lambda)+\#\mathcal E(\psi_\lambda))$ basic operations using quadrature, where $c$ depends only on the polynomial degrees $d_1$ and $d_2$. Otherwise, we evaluate 
  \[
    \langle \xi, \psi_\lambda\rangle = \sum_{\psi_{\mu}\in \mathcal R(\psi_\lambda)} h_{\lambda,\mu} \langle \xi, \psi_{\mu}\rangle,
  \]
  which requires $\#\mathcal R(\psi_\lambda)$ operations. 
  It remains to estimate the size of 
  \[
    \Theta_{\xi} = \{\lambda' \in \Theta: \psi_{\lambda'}\in \mathcal R(\psi_\lambda) \text{ for $\lambda\in \Theta$  with $\xi$ not polynomial on $K\in\mathcal K(\psi_\lambda)$ or $E\in\mathcal E(\psi_\lambda)$}\}.
  \]
  Note that if $\xi$ is not polynomial on all  $K\in\mathcal K(\psi_\lambda)$ and $E\in \mathcal E(\psi_\lambda)$, then $\psi_\lambda\notin \Theta_{0}(\cT)$ by definition. Hence, 
  \[
    \Theta_{\xi}\subseteq \{\lambda' \in \Theta: \psi_{\lambda'}\in \mathcal R(\psi_\lambda) \text{ for some } \lambda\in \Theta\setminus\Theta_{0}(\cT) \}
  \]
  and $\#\Theta_{\xi}\leq  \#(\Theta\setminus\Theta_{\xi,0} )\max_{\lambda\in\Theta}\#\mathcal R(\psi_\lambda) \leq \#(\Theta\setminus\Theta_{J}(\cT))\max_{\lambda\in\Theta}\#\mathcal R(\psi_\lambda)$. The number of basic operations thus does not exceed
  \[
    c\, (\#\mathcal K(\psi_\lambda)+\#\mathcal E(\psi_\lambda)+\max_{\lambda\in\Theta}\#\mathcal R(\psi_\lambda))\,\left(\#\Theta_{\xi}+ \#\Theta\setminus\Theta_{J}(\cT)\right)\leq C \#\Theta\setminus\Theta_{J}(\cT). \qedhere
  \]
\end{proof}

A practical method implementing this result is given with the method \textsc{DualNormIndicators} in \Cref{alg: BPX-coeffs}.
Finally, this and the previous algorithms are used to estimate the residual similarly as in~\cite{BV} with \Cref{alg: ResEstimate} for prescribed tolerances of the approximations.

\begin{algorithm}[t]
	\caption{$(\Theta^+, (\langle \xi, \psi_\lambda\rangle)_{\lambda \in \Theta^+} ) = \text{\textsc{DualNormIndicators}}(\xi,J)$ with $\xi$ as in~\eqref{eq:xirepresentation}, $J\in \N$.}\label{alg: BPX-coeffs}
  \begin{flushleft}
    Let $\mathcal T$ and $\mathcal E$ be as in \Cref{lmm:spatialrelerr}, with $\cT$ the smallest triangulation such that~\eqref{eq:xirepresentation} holds for $\xi$.

    Initialize $L = 0$ and $\tilde\Theta_L = \{\lambda\colon \abs{\lambda} =0 \}$.
  \end{flushleft}
  \begin{enumerate}[{\bf(i)}]
      \item\label{algDualNormloop1} Set $L\leftarrow L +1$ and 
      
      $
      \displaystyle\tilde\Theta_{L}= \Big\{\lambda'\colon \psi_{\lambda'} \in \mathcal R(\psi_\lambda) \text{ for $\lambda\in \tilde\Theta_{L-1}$ $\wedge$ } $

        \quad\quad\quad\quad\text{$\displaystyle\Big( \,\big(\xi$ not polynomial on $K$ or $E$ for some $K\in \mathcal K(\psi_\lambda)$, $E\in \mathcal E(\psi_\lambda)\big)$} 

        \quad\quad\quad\quad$\displaystyle\text{ $\vee$ $\big(\supp \psi_{\lambda'}\cap K'\neq \emptyset $ and $ L \leq \level(K')+J$
        for some $K'\in \mathcal K\big)$}$

        \quad\quad\quad\quad$\displaystyle\text{ $\vee$ $\big(\supp \psi_{\lambda'}\cap E'\neq \emptyset $ and $ L \leq \level(E')+2 J$
        for some $E'\in \mathcal E\big)\,\Big)$}\Big\};
      $
      \item If $\tilde\Theta_L\neq \emptyset$ go to \textbf{(i)}, else go to \textbf{(iii)};
    \item For $j = L-1,L-2,\dots, 0$: 
    For all $\lambda\in\tilde\Theta_j$ evaluate $\langle \xi, \psi_\lambda\rangle$ via
        
    \[ \displaystyle   \langle \xi, \psi_\lambda\rangle = \sum_{\psi_{\lambda'}\in \mathcal R(\psi_\lambda)} r_{\lambda,\lambda'} \langle \xi, \psi_{\lambda'}\rangle\quad
    \text{if $\mathcal R(\psi_\lambda)\subset \tilde\Theta_{j+1}$,}\qquad \text{or by quadrature otherwise;}
        \]
    \item Set $\displaystyle\Theta^+ = \bigcup_{j=0}^{L-1} \tilde\Theta_j$ and return  $\displaystyle \bigl( \Theta^+,\, \bigl(\langle \xi, \psi_\lambda\rangle\bigr)_{\lambda \in \Theta^+}\bigr)$.
\end{enumerate}
\label{dualnormindicators}
\end{algorithm}

\begin{remark}
  For computational purposes one can avoid explicit treatment of edge terms: instead of applying integration by parts as in~\eqref{eq:intbyparts} to each term of the form $A_\mu v_\nu$, one can also use for each given $\psi_\lambda$ that
  \[
    \langle A_\mu v_\nu, \psi_\lambda\rangle = \int_D q_{\mu, \nu} \cdot \nabla \psi_\lambda\sdd x 
  \]
  with a piecewise polynomial vector field $q_{\mu,\nu}$. 
  One can thus apply the same algorithmic considerations to the components of these vector fields on the triangles and then form inner products with the gradients of the frame elements.
\end{remark}

\begin{algorithm}[htp]
  \caption{$((\Theta_\nu^+)_{\nu\in F^+},(\hat{\br}_\nu)_{\nu\in F^+},([r] _\nu)_{\nu\in F^+}, \eta, b) \!=\! \text{\textsc{ResEstimate}}(v; \zeta, \eta_0, \varepsilon)$, for $\#\supp ([v]_\nu)_\nu < \infty$, relative tolerance $\zeta>0$, initial tolerance $\eta_0$, target tolerance $\varepsilon$.}\label{alg: ResEstimate}

  \flushleft Set $\eta = \eta_0$; choose $\hat J$ such that $\zeta_{\hat J} := C 2^{-\hat J} < \zeta$.
  \begin{enumerate}[{\bf(i)}]
    \item Set $( M(\nu) )_{\nu \in F^+}, (F_i, \ell_i)_{i = 0}^I = \text{\textsc{Apply}}(v; \eta/C_\Psi)$ by \Cref{apply_semidiscr};
    
  \item For each $\nu' \in \supp([v]_\nu)_\nu$
  
  \noindent
  \quad Let $\ell_{\nu'} = \max_{\nu'\in F_i}\ell_i$;
  
  \noindent
  \quad For each $\mu$ such that $\abs{\mu}\leq \ell_{\nu'}$ and $\nu$ such that $(\mu, \nu') \in M(\nu)$\vspace{6pt} compute 

  \quad\quad $\displaystyle  [g]_{\nu,\nu'} = (\bM_{\mu})_{\nu,\nu'} A_\mu [v]_{\nu'}$
  
  \noindent
  \quad by traversing the mesh of $[v]_{\nu'}$ once;

  \item For each $\nu \in F^+$, use \Cref{alg: sum} to evaluate
  
  \quad\quad $\displaystyle [r] _\nu = \delta_{0,\nu} f - \sum_{(\mu,\nu')\in M(\nu)} [g]_{\nu,\nu'} = \text{\textsc{Sum}}\bigl(\delta_{0,\nu} f, (-[g]_{\nu,\nu'})_{(\mu,\nu')\in M(\nu)}\bigr)$;
   
   \item For each $\nu \in F^+$ set $(\Theta_\nu^+,\hat\br_\nu) = \text{\textsc{DualNormIndicators}}( [r]_\nu,\hat J)$ by \Cref{alg: BPX-coeffs};
    \item Let $b = \eta + (1+\frac{\zeta_{\hat J}}{\sqrt{1-\zeta_{\hat J}^2}})\norm{(\norm{\hat\br_\nu}_{\ell_2})_{\nu\in F^+}}$. If $\eta \leq \frac{\zeta - \zeta_{\hat J}}{ 1+\zeta} \norm{(\norm{\hat\br_\nu}_{\ell_2})_{\nu\in F^+}}$ or $b \leq \varepsilon$,

    \noindent
    \quad return $((\Theta_\nu^+)_{\nu\in F^+},(\hat\br_\nu)_{\nu\in F^+},([r] _\nu)_{\nu \in F^+}, \eta, b)$; 
    
    \noindent otherwise, set $\eta \leftarrow \eta/2$ and go to  {\bf(i)};
  \end{enumerate}
  \label{optscheme}
  \end{algorithm}

Note that with sequences $(\br_\nu)_{\nu\in \cF}$ such that $\br_\nu \in \ell_2(\Theta)$ for $\nu \in \cF$ and $\sum_{\nu \in \cF} \norm{\br_\nu}_{\ell_2}^2 < \infty$, we associate $\br = (\br_{\nu,\lambda})_{\nu\in\cF,\lambda\in\Theta} \in \ell_2(\cF\times\Theta)$. We write $\norm{\br} = \norm{\br}_{\ell_2}$, so that in particular
\[  \norm{\br} = \bignorm{\bigl(\norm{\br_\nu}_{\ell_2(\Theta)}\bigr)_{\nu \in\cF}}_{\ell_2(\cF)}. \]

  \begin{prop}\label{prop:ResEstimate}
    Let $((\Theta_\nu^+)_{\nu \in F^+}, (\hat\br_
    \nu)_{
      \nu\in F^+
    }, \eta, b)$ be the return values of Algorithm \ref{optscheme} and let
\begin{equation}\label{eq:Lambdaplus}
  \Lambda^+ = \bigl\{  (\nu, \lambda) \in \cF \times \Theta\colon \nu \in F^+, \lambda \in \Theta_\nu^+ \bigr\}\,.
\end{equation} 
    Set $\hat\br_\nu = 0$ for $\nu\notin F^+$.
    Furthermore, let $\bz = (\langle [Bv]_\nu, \psi_\lambda\rangle)_{\nu\in\cF,\lambda\in\Theta}$ and $\bbf = (\langle [f]_\nu, \psi_\lambda\rangle)_{\nu\in\cF,\lambda\in\Theta}$.
    Then $\norm{\bz - \bbf}\leq  b$ and either $b \leq \varepsilon$, or $\hat\br$ satisfies 
    \begin{equation}\label{eq:reserrbound} 
      \norm{\hat\br - (\bbf - \bz)} \leq \zeta \norm{\bbf - \bz},
    \end{equation}
    where we have
    $ \#\supp\hat\br \leq \#\Lambda^+ = \sum_{\nu \in F^+} \# \Theta_\nu^+$  and
    \begin{multline}\label{rsuppest}
     \#\Lambda^+ \lesssim 
     \#\cT(f)+
     (1+\abs{\log \eta} +{\log\norm{(\norm{[v]_\nu})_{\nu\in \cF}}_{\mathcal A^s}})^2 N(\mathbb{T})\\
     + 
     (1+\abs{\log \eta} +{\log\norm{(\norm{[v]_\nu})_{\nu\in \cF}}_{\mathcal A^s}})\eta^{-\frac{1}{s}}\norm{(\norm{[v]_\nu})_{\nu\in \cF}}_{\mathcal A^s}^{\frac{1}{s}}\,.
    \end{multline}
    The number of operations in \Cref{alg: ResEstimate} is bounded by a fixed multiple of
    \begin{multline}\label{ropest}
      (1+\log_2(\eta_0/\eta))\Big(
      \#F\log\#F+\#\cT(f)+
      (1+\abs{\log \eta} +{\log\norm{(\norm{[v]_\nu})_{\nu\in \cF}}_{\mathcal A^s}})^2 N(\mathbb{T}) \\
      + 
      (1+\abs{\log \eta} +{\log\norm{(\norm{[v]_\nu})_{\nu\in \cF}}_{\mathcal A^s}})\eta^{-\frac{1}{s}}\norm{(\norm{[v]_\nu})_{\nu\in \cF}}_{\mathcal A^s}^{\frac{1}{s}}\Big).
    \end{multline}
    \end{prop}
    \begin{proof}
      First, we show the residual error bounds similarly to \cite[Thm. 4.15]{BV}. As a consequence of \Cref{lm:frame} and \Cref{semidiscrapply}, we have $\norm{\bz- \bg }\leq C_{\Psi} \norm{Bv - g}_{\cV'}\leq \eta$, where with $[g]_\nu$ as in \eqref{eq:semidiscrwdefM},
      \[
        \bg_\nu = \bigl(\langle [g]_\nu, \psi_\lambda \rangle\bigr)_{\lambda\in\Theta}\quad\text{and }\quad
        \bigl(\langle [g]_\nu - \delta_{0,\nu} f, \psi_\lambda \rangle\bigr)_{\lambda\in\Theta_\nu} = \hat\br_\nu.
      \]
      By \Cref{lmm:spatialrelerr}, we can bound the truncation error by 
      \[
      \sum_{\nu\in F^+}\sum_{\lambda \notin \Theta_\nu}  \abs{\langle [g]_\nu - \delta_{0,\nu} f, \psi_\lambda \rangle}^2 
      \leq \zeta_{\hat J}^2
      \sum_{\nu\in F^+}\sum_{\lambda \in\Theta}  \abs{\langle [g]_\nu - \delta_{0,\nu} f, \psi_\lambda \rangle}^2
      .
      \]
      Thus
      \[
        \norm{\hat\br}^2 = 
       \sum_{\nu\in F^+}\sum_{\lambda \in \Theta_\nu}  \abs{\langle [g]_\nu - \delta_{0,\nu} f, \psi_\lambda \rangle}^2 
        \geq
        (1-\zeta_{\hat J}^2) 
        \sum_{\nu\in F^+}\sum_{\lambda \in\Theta}  \abs{\langle [g]_\nu - \delta_{0,\nu} f, \psi_\lambda \rangle}^2
        = (1-\zeta_{\hat J}^2) \norm{\bbf - \bg}^2.
      \]
      Together with $\norm{\bbf -\bz} \geq \norm{\bbf - \bg} -\eta$, this first results in
      \begin{equation}\label{eq:relative residual error}
        \norm{\hat\br -(\bbf-\bz)}\leq \norm{\bw-\bz} + \norm{\hat\br-(\bbf-\bg)}
        \leq \eta + \zeta_{\hat J} \norm{\bbf-\bg} 
        \leq 
        \eta + \frac{\zeta_{\hat J}}{\sqrt{1-\zeta_{\hat J}^2}}\norm{\hat\br}
      \end{equation}
      and hence 
      $\norm{\bbf-\bz} \leq \eta + (1+{\zeta_{\hat J}}(1-\zeta_{\hat J}^2)^{-\frac12})\norm{\hat\br} = b$.
      If additionally $\eta \leq \frac{\zeta-\zeta_{\hat J}}{1+\zeta}\norm{\hat\br}$, it results in
      \[
        \norm{\hat\br -(\bbf-\bz)}
        \leq
        \eta + \zeta_{\hat J} \norm{\bbf-\bg} 
        \leq (\zeta-\zeta_{\hat J})\norm{\hat\br} + \norm{\bbf-\bg}
        - \zeta\eta
        \leq
        \zeta(\norm{\bbf-\bg} -\eta)
        \leq \zeta \norm{\bbf-\bz}.
      \]  
      This shows the prescribed error accuracy.
      
      We now estimate $\#\Lambda^+ = \sum_{\nu \in F^+} \# \Theta_\nu^+$. Let $\tilde\cT_\nu\geq \hat\cT_0$ be a triangulation, such that $[r]_\nu$ is polynomial on its triangles and edges.  Then by \Cref{lm: Triangulation to frame elements} and \Cref{lm: BPXcoeffscosts}, we have
      \[
        \# \Theta_\nu^+ \leq c(\hat J) \#\tilde\cT_\nu. 
      \]
      We thus estimate $\sum_{\nu\in F^+} \#\tilde\cT_\nu$. To this end, we also denote by
      \[
        \tilde\cT_{\nu,\nu'} = \{T\in \tilde\cT_\nu\colon \supp[g]_{\nu,\nu'} \cap  T \neq \emptyset \}
      \]
      the supporting triangles of $[g]_{\nu,\nu'}.$
      Let $T$ be a triangle in the triangulation $\cT_{\nu'}$ of $[v]_{\nu'}$. If $\level(T) =  \ell+k$ for some $\ell\leq \ell_{\nu'}$, then all $\theta_\mu$ with $\abs\mu\leq \ell$ are polynomial on $T$ by \Cref{ass:pwpolytheta}. On the one hand, by \Cref{ass:wavelettheta}~(ii) we have
      \[
      \#\bigl\{(\tilde T,\nu)\colon \tilde T \in\hat\cT_{\nu,\nu'} \text{ for some $\nu = \nu'\pm e_\mu$, $\abs\mu\leq \ell$ and $ T\cap \tilde T\neq \emptyset$}\bigr\} \leq 2M\ell\leq 2M \ell_{\nu'}.
      \] 
      On the other hand, $
        \#\{\tilde T\in \cT_{\nu'\pm e_\mu,\nu'}\colon \level(\tilde T) = \abs\mu+k   \}\leq 2K $
      by \Cref{ass:pwpolytheta}. Hence, we have 
      \begin{equation}\label{eq:MeshEstimate}
        \sum_{\abs\mu\leq \ell_\nu}\left(\#\cT_{\nu'+e_\mu,\nu'} +\#\cT_{\nu'-e_\mu,\nu'}\right)
        \leq 
         2M\ell_{\nu'}\#\cT_{\nu'}+\sum_{\abs{\mu}\leq \ell_{\nu'}}2K
         \lesssim \#\cT_{\nu'}\ell_{\nu'} + 2^{d\ell_{\nu'}}
      \end{equation}
      and 
      \begin{align*}
        \#\Lambda^+ &= \sum_{\nu \in F^+} \# \Theta_\nu^+
        \lesssim
        \#\cT(f)+
        N(\mathbb{T})\, \max_{\nu'\in F}\ell_{\nu'}^2 
        + \sum_{j = 0}^J\#F_j  2^{d\ell_{j}}
        \max_{\nu'\in F}\ell_{\nu'}\,,
      \end{align*}
      where an additional factor of $\max_{\nu'\in F}\ell_{\nu'}$ occurs by completing the supports $\tilde\cT_{\nu,\nu'}$ to a (not necessarily conforming) triangulation. The estimate~\eqref{rsuppest} then follows with \Cref{semidiscrapply}.

      It remains to estimate the number of operations.
      As in \cite{BV}, the cost of \Cref{apply_semidiscr} is of order
      \[
        N(\mathbb{T}) +\#F\log{\#F} + 
        \eta^{-\frac{1}{s}} 
        \norm{(\norm{[v]_\nu})_{\nu\in \cF}}_{\mathcal A^s}^{\frac{1}{s}}.
      \]
      For step (ii) in \Cref{alg: ResEstimate}, we have to count how often each triangle in $\T$ is required, leading to an order of 
      \[
        (1+\abs(\log \eta + \log{\norm{(\norm{[v]_\nu})_{\nu\in \cF}}_{\mathcal A^s}})) N(\mathbb{T})
        + \eta^{-\frac{1}{s}} 
        \norm{(\norm{[v]_\nu})_{\nu\in \cF}}_{\mathcal A^s}^{\frac{1}{s}}.
      \]
      The application of \textsc{Sum} in \Cref{alg: sum} is linear in the size of triangulations. This results in the complexity 
      \begin{multline*}
        \#\cT(f)+
         (1+\abs{\log \eta + \log{\norm{(\norm{[v]_\nu})_{\nu\in \cF}}_{\mathcal A^s}}})^2 N(\mathbb{T}) \\
        + 
        \eta^{-\frac{1}{s}} 
        \norm{(\norm{[v]_\nu})_{\nu\in \cF}}_{\mathcal A^s}^{\frac{1}{s}}(1+\abs{\log \eta} + \log{\norm{(\norm{[v]_\nu})_{\nu\in \cF}}_{\mathcal A^s}}).        
      \end{multline*}
      The number of basic operations in \Cref{alg: BPX-coeffs} is linear in the size of $\Lambda^+$, which is again of the same order by \Cref{lm: BPXcoeffscosts}.
      Finally, we have at most $1+\log_2(\eta_0/\eta)$ outer loops, which leads to the bound~\eqref{ropest}.
    \end{proof}

\begin{remark}\label{rem:generalf}
For simplicity, as stated in the beginning of Section \ref{sec:frames}, so far we have assumed the parameter-independent source term $f$ to be piecewise polynomial on the initial triangulation $\initmesh$. However, the above procedure for computing residual indicators can easily be modified to accommodate more general $f$. 
There are several possible ways of achieving this by straightforward adaptation of standard strategies. To avoid additional technicalities, we restrict ourselves to briefly outlining two options:
 \begin{enumerate}[{\rm(a)}]
\item  A general approach is to use an additional routine $\text{\sc Approx}_f(\eta)$ that returns an approximation $\tilde f$ of $f$ satisfying $\norm{\tilde f -f}_{H^{-1}(D)}\leq \eta$ and is of the form as $\xi$ in \Cref{lmm:spatialrelerr} for some triangulation. 
 This approximation can be used in step~{\bf(iii)} of \Cref{alg: ResEstimate}.
 Then we can guarantee the result of \Cref{prop:ResEstimate} again with the help of \Cref{lmm:spatialrelerr}. A very similar strategy was used in \cite[Sec.~6]{S07}. The construction of $\text{\sc Approx}_f(\eta)$ depends on the particular features of $f$: if $f$ is in $L_2(D)$, such a routine can be obtained by $L_2$-projections onto a suitable space of piecewise polynomials; for proper functionals $f\in H^{-1}(D)$, such approximations are more problem-specific, as considered in detail in \cite{CDN:12}. 
\item We can also more directly use knowledge of the coefficients $\langle f, \psi_\lambda\rangle$, $\lambda \in \Theta$ and their decay to provide estimates \[ \Bigl(\sum_{\lambda\in \Theta\setminus \Theta^f_\eta} \abs{\langle f, \psi_\lambda \rangle}^2 \Bigr)^{\frac{1}{2}}\leq \eta \] for certain finite subsets $\Theta^f_\eta\subset \Theta$. These can be used as input in step~{\bf(iv)} of \Cref{alg: ResEstimate}.
\end{enumerate}
\end{remark}

\subsection{Refining the triangulations}\label{sec:refinetriangulation}
Our refinement strategy is based on selecting a subset of the residual indicators produced by \textsc{ResEstimate} (\Cref{alg: ResEstimate}) according to a bulk chasing criterion and subsequently refining the individual meshes in the approximation such that they resolve the selected frame elements. We assume that \textsc{ResEstimate} is performed for an approximation on the given conforming meshes $\mathbb{T} = (\cT_\nu)_{\nu \in F^0}$ with finite $F^0 \subset \cF$.

Let us first consider the selection of residual indicators. 
Recall that for all $\nu \in F^+ \subset \cF$, \textsc{ResEstimate} produces vectors of residual indicators $\hat\br_\nu$ corresponding to indices $\Theta_\nu^+$, with associated spatial-parametric index set $\Lambda^+ \subset \cF \times \Theta$ as in \eqref{eq:Lambdaplus}. Since the computational costs of the operations that we perform depend on tree structure in the frame index sets, we use the strategy based on tree coarsening described in \cite[Sec.~4.5]{BV} that preserves such structures in the selection.

To this end we first fix a tree structure on the frame elements $\psi_\lambda$, $\lambda \in \Theta$, and thus on the index set $\Theta$. 
This tree structure is determined by choosing a unique parent $\psi_{\lambda'}$ for each frame element $\psi_\lambda$ with $\abs{\lambda}>0$ such that $|\lambda| = |\lambda'| + 1$ and $\meas_2 (\supp \psi_{\lambda } \cap \supp \psi_{\lambda'} ) > 0$.
The tree structure on $\Theta$ induces a natural tree structure on $\Theta^\cF$ (and thus on $\cF\times \Theta$) with roots $(\{\psi_\lambda \colon |\lambda|=0\})_{\nu\in \cF}$. 

We use the procedure $\Lambda = \textsc{TreeApprox}(\Lambda^0, \Lambda^+, \hat\br, \eta)$ from \cite[Alg.~4.4]{BV} to obtain a coarsening $\Lambda$ of $\Lambda^+$ such that $\Lambda^0\subseteq\Lambda\subseteq \Lambda^+$, where $\Lambda^0, \Lambda$ are subsets with the chosen tree structure and $\eta>0$; see also \cite{BD:04, B:18}. Here we take $\Lambda^0 = \{ (\nu,\lambda) \in F^0\times \Theta \colon \psi_\lambda \in V(\cT_\nu)\}$.

For $\Lambda$ generated in this manner, $\textsc{TreeApprox}$ ensures
\[
  \norm{\hat\br|_{\Lambda}}^2 \geq \norm{\hat\br}^2 - \eta.
\]
For any prescribed $\omega_0\in(0,1)$, taking $\eta = (1-\omega_0^2)\norm{\br}^2$ we obtain the desired bulk chasing condition
\begin{equation}\label{eq:bulkchasing}
  \norm{\hat\br|_{\Lambda}} \geq \omega_0\norm{\hat\br} \,.
\end{equation}

\begin{remark}
  As a consequence of \cite[Cor.~4.20]{BV}, the resulting $\Lambda$ has the following optimality property: for each $\omega_1 \in [\omega_0, 1)$, there exists $\tilde C>0$ such that $\#(\Lambda\setminus \Lambda^0) \leq \tilde C \#( \tilde\Lambda\setminus\Lambda^0)$ for all tree subsets $\tilde\Lambda\supseteq \Lambda^0$ such that $\norm{\hat\br|_{\tilde\Lambda}} \geq \omega_1 \norm{\hat\br}$.
\end{remark}

For $\Lambda$ satisfying \eqref{eq:bulkchasing} selected in this manner, let $F = \{ \nu\in\cF\colon \exists \lambda \in\Theta\colon (\nu,\lambda)\in\Lambda\} \subseteq F^+$ and $\Theta_\nu = \{ \lambda\in \Theta\colon (\nu,\lambda) \in\Lambda\} \subseteq \Theta_\nu^+$ for $\nu\in F$, where $\Theta_\nu$ inherits the tree structure of $\Lambda$.
We now define a procedure 
\begin{equation}\label{eq:refmesh}
  (\tilde\cT_\nu)_{\nu\in F} = \msx{\Lambda}
\end{equation}
  that outputs the componentwise smallest sequence of meshes such that for each $\nu \in F$, $\tilde\cT_\nu$ is conforming and $\Span \{ \psi_\lambda\}_{\lambda \in \Theta_\nu}\subseteq V(\tilde\cT_\nu)$. Note that in a setting where $\Lambda^0 \subseteq \Lambda$, the meshes $\tilde\cT_\nu$ are refinements of the initial meshes $\cT_\nu$ that are given by the frame elements in $\Lambda^0$, that is, $(\cT_\nu)_{\nu\in F^0} =\msx{\Lambda^0}$.

\begin{prop}\label{prop:meshcomplexity}
The meshes $(\tilde\cT_\nu)_{\nu\in F}$ produced by $\msx{\Lambda}$ in \eqref{eq:refmesh} satisfy $\#\tilde\cT_\nu \lesssim \#\Theta_\nu$ for each $ \nu \in F$ and can be obtained using a number of operations proportional to $\#\Theta_\nu$, with constants depending only on $\hat\cT_0$.
\end{prop}

The proof relies on the following bound on the complexity of conforming meshes created by newest vertex bisection, which is a consequence of \cite[Thm.~2.4]{BDD}, see also \cite[Thm.~3.2]{S07} and \cite[Lem.~2.3]{CKNS:08}.

\begin{theorem}\label{thm:nvbisect}
 Let $\cT_0 = \hat \cT_0$ and for $k = 1,\ldots,n$, let $\cT_k$ be defined as the smallest conforming refinement of $\cT_{k-1}$ such that the elements $\cM_{k-1} \subseteq \cT_{k-1}$ are bisected. Then with $C_0>0$ depending only on $\hat\cT_0$, 
 \[
    \# \cT_n - \#\cT_0 \leq C_0 \sum_{k=0}^{n-1} \# \cM_k. 
 \]
\end{theorem}

\begin{proof}[Proof of Proposition~\ref{prop:meshcomplexity}]

For each $\nu \in F$, given a tree $\Theta_\nu$, we can apply Theorem~\ref{thm:nvbisect} to the construction of a suitable triangulation $\tilde \cT_\nu$ with $\# \tilde \cT_\nu \lesssim \# \Theta_\nu$
in the following way.
We may assume that $\Theta_\nu$ contains the roots $\{\lambda\colon \abs{\lambda} = 0\}$ as 
\[\#(\{\lambda : \text{$\lambda$ is an ancestor of some $\lambda'\in \Theta_\nu$ }\} \cup \Theta_\nu)
\lesssim
\#\Theta_\nu
\]
by the tree structure of $\Theta_\nu$.
Now let the sequence of triangulations in \Cref{thm:nvbisect} be defined by $\mathcal T_0 = \hat{\mathcal T}_0$
and 
\[
\mathcal M_k = \{T\in \mathcal T_k\colon T\cap \supp{\psi_\lambda} 
\text{ and } 
\abs{\lambda} = k+1
\}
\] 
with the slight adaptation that each marked triangle is bisected twice.
We get $\sum_{k=0}^{n-1} \# \cM_k \lesssim \Theta_\nu$ as each hat function on a uniform refinement has support on only a bounded number of triangles, and thus the result follows.
\end{proof}

\subsection{Solving Galerkin systems}\label{sec: galerkin solver}

\begin{algorithm}[t]
	\caption{$w = \text{\textsc{GalerkinSolve}}(\T,v, r, \ell, \varepsilon_0)$ given a family of triangulations $\T$, an approximate solution $v \in \cV({\T})$, an approximated residual $r \in \cV'$, accuracy parameter $\ell$ of the operator compression, and a target tolerance $\epsilon_0$.}
  \label{alg:galerkin solve}
\begin{enumerate}[{\bf(i)}]
\item Assemble linear operators $\bB_\ell$, $\bP$ on the coefficients of the nodal basis of $\T$ such that
      $\bu^\intercal \bB_\ell \bv  =  \langle B_\ell v, u\rangle$ for all $u,v\in \cV(\T)$, $\bP$ satisfies~\eqref{eq: spectral equivalence}, and $\br$ such that $\langle \bv, \br \rangle = \langle r, v\rangle$.
\item Use the preconditioned conjugated gradient method to find $q$ with nodal coefficients $\bq$ and $\langle \bP (\bB_\ell\bq-\br), (\bB_\ell\bq-\br) \rangle \leq \varepsilon^2_0$.
\item Return $w = v + q$;
\end{enumerate}
\end{algorithm}

As a starting point for solving the Galerkin systems, we recall the frame property \Cref{lm:frame} and assume an approximated residual $r$ computed by \Cref{optscheme} with tolerance $\eta$, which satisfies
$
  \norm{r- (f- Bv)}_{\cV'}\leq {\eta}/{C_\Psi}
$.
We use the preconditioned conjugated gradient~(PCG) method to find a correction $q$ of the current solution approximation such that for a $\rho>0$,
\begin{equation}\label{eq:updatebound}
  \|r-B_\ell q\|_{\cV(\T)'}\leq \frac{\rho}{c_\Psi}\norm{\hat \br}.
\end{equation}
Here, the discrete dual norm is given by
\[
  \norm{g}_{\cV({\T})'} = \sup_{v\in\cV({\T})} \frac{\langle g, v\rangle}{\norm{v}_{\cV}}.
\]
The condition \eqref{eq:updatebound} can be verified directly in the PCG method as follows. Let $\bA \in \R^{N\times N}$ with $N= \dim \cV(\T)$ be the coefficient matrix such that for nodal coordinate vectors $\bu, \bv \in \R^N$ of functions $u, v \in \cV(\T)$ we have $\langle \bA\bu, \bv \rangle = \int_Y\int_D \nabla u(x,y)\cdot\nabla v(x,y) \sdd x \sdd y$. We use the modification of the BPX preconditioner to general meshes generated by newest vertex bisection analyzed in \cite{CNX:12}; an alternative multigrid approach is presented in \cite{WZ:17}.
For this preconditioner with matrix representation $\bP$, there exist $c_p, C_P>0$ such that we have the spectral equivalence
\begin{equation}\label{eq: spectral equivalence}
  c_P \langle \bP^{-1} \bu, \bu\rangle \leq \langle\bA \bu , \bu\rangle
  \leq C_P \langle \bP^{-1} \bu, \bu\rangle \quad\text{for all $\bu \in \R^N$.}
\end{equation}
For $g$ in $\cV'$, we thus have
\[
  \norm{g}_{\cV({\T})'}^2 = \sup_{v\in\cV({\T})} \frac{\langle g, v\rangle ^2}{\norm{v}_{\cV}^2}
  =\sup_{\bv\in \R^N} \frac{\abs{\langle\bg, \bv\rangle}^2}{\langle \bA \bv, \bv\rangle}
  \leq \sup_{\bh\in \R^N}\frac{\abs{\langle\bg, \bP\bh\rangle}^2  }{c_P \langle\bh , \bP \bh\rangle }\leq \frac{1}{c_P} \langle \bP \bg, \bg\rangle,
\]
where we substituted $\bv = \bP \bh$, used the spectral equivalence and the Cauchy-Schwarz inequality.
Similarly, we have $\norm{g}_{\cV({\T})'}^2 \geq \frac{1}{C_P} \langle \bP\bg, \bg\rangle$.

\begin{prop}\label{prop: galerkin accuracy}
  Assume an initial approximation $v$ on a triangulation $\tilde\T\leq \T$, and with sufficiently small global tolerance $\varepsilon$, let 
  \[
  ((\Theta_\nu^+)_{\nu\in F^+},(\hat{\br}_\nu)_{\nu\in F^+},([r] _\nu)_{\nu\in F^+}, \eta, b) \!=\! \text{\textsc{ResEstimate}}(v; \zeta, \eta_0, \varepsilon)
  \]
  be the output of \Cref{alg: ResEstimate} and 
  $w = \text{\textsc{GalerkinSolve}}(\T,v, r, \ell,{c_P}^{-1/2}c_\Psi^{-1}\rho\norm{\hat\br})$ be the output of 
  \Cref{alg:galerkin solve}. Then the approximation $w$ of the Galerkin solution $u_\T$ satisfies the error bound
   \[
      \norm{u_\T - w}_B \leq \gamma(\zeta, \ell, \rho, \hat J)\norm{\hat\br},
    \]
    where
    \[
      \gamma(\zeta, \ell, \rho, \hat J) = \frac{1}{c_\Psi}\frac{1}{\sqrt{c_B}}\left( \frac{\zeta-\zeta_{\hat J}}{1+\zeta} + \frac{\zeta_{\hat J}}{\sqrt{1-\zeta_{\hat J}^2}} + \rho + \frac{C_2}{1 - 2^{-\alpha}} 2^{-\alpha \ell} \frac{1}{c_B}\left(\frac{1}{1-\zeta_{\hat J}} + \rho\right) \right).
    \]
\end{prop}

\begin{proof}
  We assume that $\varepsilon$ is sufficiently small so that \textsc{ResEstimate} does not terminate with the condition $b\leq \varepsilon$ in step \textbf{(v)}.
  Recall from~\eqref{eq:dual norm equivalence} that $\norm{v}_B \leq c_B^{-1/2}\norm{Bv}_{\cV(\T)'}$ for $v \in \cV(\T)$.
  We consider a decomposition of the total error for $w = v + q$,
  \begin{equation}\label{eq:galerrordecomp}
    \begin{aligned}
    \norm{ u_\mathbb{T} - w  }_B &\leq \frac{1}{\sqrt{c_B}} \norm{ f - B (v + q)}_{\mathcal{V}(\T)'}\\   
     &\leq \frac{1}{\sqrt{c_B}} \left( \norm{ r - (f - Bv)  }_{\mathcal{V}(\mathbb{T})'} 
       + \norm{ r - B_\ell q }_{\mathcal{V}(\mathbb{T})'} 
        + \norm{ (B_\ell - B) q  }_{\mathcal{V}(\T)'} \right).
    \end{aligned}
  \end{equation}
  For the first term on the right in \eqref{eq:galerrordecomp}, using~\eqref{eq:relative residual error} in the proof of \Cref{prop:ResEstimate} and \Cref{lm:frame}
  \begin{equation*}
    \norm{ r - (f - Bv)  }_{\mathcal{V}(\mathbb{T})'} \leq \frac{1}{c_\Psi} \left(\frac{\zeta-\zeta_{\hat J}}{1+\zeta} + \frac{\zeta_{\hat J}}{\sqrt{1-\zeta_{\hat J}^2}}\right)\norm{\hat\br}.
  \end{equation*}
  Note that \Cref{prop:ResEstimate} uses the frame norm while the estimate above is with respect to the discrete dual norm, which is obviously smaller than the full dual norm used in \Cref{lm:frame}.

  By assumption, the second term on the right in \eqref{eq:galerrordecomp} satisfies the bound \eqref{eq:updatebound}
  with relative solver error $\rho$ as a result of \Cref{alg:galerkin solve}.
  For the last term, first note that
  \begin{equation*}
    \norm{B_\ell q}_{\mathcal{V}(\mathbb{T})'} = \norm{r - (r - B_\ell q)}_{\mathcal{V}(\mathbb{T})'} \leq \norm{r}_{\cV(\mathbb T)'} + \frac{\rho}{c_\Psi}\norm{\hat\br}.
  \end{equation*}
  Moreover, again with \Cref{prop:ResEstimate} for the approximation of the full sequence of residual frame coefficients $\bar\br = ((\langle [r]_\nu, \psi_\lambda\rangle)_{\lambda\in\Theta})_{\nu\in\cF}$, we derive
  \begin{equation}\label{eq:r0 estimate}
    \norm{r}_{\cV(\mathbb T)'} \leq \frac{1}{c_\Psi}\norm{\bar\br} \leq \frac{1}{c_\Psi}\frac{1}{1-\zeta_{\hat J}}\norm{\hat\br}.
  \end{equation}
  By \Cref{prop:operator compression} and using that $c_B \norm{q}_{\cV(\mathbb T)} \leq \norm{B_\ell q}_{\mathcal{V}(\mathbb{T})'}$, we thus obtain 
  \begin{equation*}
    \norm{ (B_\ell - B) q }_{\mathcal{V}(\mathbb{T})'} \leq \frac{C_2}{1 - 2^{-\alpha}} 2^{-\alpha \ell} \norm{q}_{\cV(\mathbb T)}
    \leq \frac{C_2}{1 - 2^{-\alpha}} 2^{-\alpha \ell} \frac{1}{c_B}\left(\norm{r}_{\cV(\mathbb T)'} + \frac{\rho}{c_\Psi}\norm{\hat\br}\right).
  \end{equation*}
  With~\eqref{eq:r0 estimate} it follows that
  \begin{equation*}
    \norm{ (B_\ell - B) q }_{\mathcal{V}(\mathbb{T})'} \leq \frac{C_2}{1 - 2^{-\alpha}} 2^{-\alpha \ell} \frac{1}{c_\Psi}\frac{1}{c_B}\left(\frac{1}{1-\zeta_{\hat J}} + \rho\right)\norm{\hat\br}.
  \end{equation*}
  A combination of the preceding estimates yields the result.
\end{proof}

\begin{remark}\label{rem: complexity gal solve}
  The number of steps of the PCG method in \Cref{alg:galerkin solve} terminates after a uniform finite number of steps since we require a fixed relative accuracy.
  Hence, the computational complexity of the algorithm does not deteriorate when successive mesh refinements take place. 
  Furthermore, the operator compression parameter~$\ell$ is chosen as a fixed value, similarly depending only on a fixed relative tolerance. The complexity of one iteration of the CG method can thus be bounded by a fixed multiple of $N(\T^k)$, and consequently the same holds true for each call of \textsc{GalerkinSolve} in \Cref{alg: adaptive method}.
\end{remark}

\begin{algorithm}[t]
	\caption{Adaptive Galerkin Method}\label{alg: adaptive method}
  \flushleft  Set the parameters $\ell, \zeta, \omega_0$ and relative accuracy $\rho$, set initial $u^0= 0$ and formally $\norm{\hat\br^{-1}} = C_\Psi\norm{f}_{\cV'}$. Let $k = 0$;
  \begin{enumerate}[{\bf(i)}]
\item \label{alg:point 2}$(\Lambda^+,(\hat\br^k_\nu)_\nu,([r^k] _\nu)_\nu, \eta_k, b_k) = \text{\textsc{ResEstimate}}(u^k; \zeta, \frac{\zeta}{1+\zeta}\norm{\hat\br^{k-1}}, \varepsilon)$ by \Cref{alg: ResEstimate} 
\item If $b_k\leq \varepsilon$, return $u^k$.
\item $\Lambda^{k+1} = \textsc{TreeApprox}(\Lambda^k, \Lambda^+, \hat\br^k, (1-\omega_0^2)\norm{\hat\br^k}^2)$;
\item $\T^{k+1} = \msx{\Lambda^{k+1}}$;
\item $u^{k+1} = \text{\textsc{GalerkinSolve}}(\T^{k+1}, u^k , r^k, \ell, \frac1{\sqrt{c_P}}\frac{\rho}{c_\Psi}\norm{\hat\br^k})$ by \Cref{alg:galerkin solve}
\item $k\leftarrow k+1$ and go to~\eqref{alg:point 2};
\end{enumerate}
\end{algorithm}

\section{Convergence analysis}
\label{sec:convergence}

We now come to the main result of this work, where we show error reduction by a uniform factor in each step of the adaptive scheme \Cref{alg: adaptive method}.
This reduction factor depends on a number of parameters and constants from previous sections: on $c_B$ and $C_B$ from \eqref{eq:uniformellipticity} and \eqref{eq:CB}, respectively; on the frame bounds $c_\Psi, C_\Psi$ in \eqref{eq:framecondition}; and $\omega_0$ as in \eqref{eq:bulkchasing}.

Let $\T = (\cT_\nu )_{\nu \in F}$ with conforming $\cT_\nu \geq \hat\cT_0$ for all $\nu \in F$ with finite $F\subset \cF$ be given, and let $\mathbb{\tilde T} = (\tilde \cT_\nu )_{\nu \in\tilde F}$ be any refinement with $\tilde F \supseteq F$ and conforming $\tilde \cT_\nu \geq \cT_\nu$ for $\nu \in F$ as well as $\tilde \cT_\nu \geq \cT_0$ for $\nu \in F^+ \setminus F$. Note that for the exact solution $u \in \cV$, the Galerkin solution $u_{\mathbb{\tilde T}} \in \cV(\mathbb{\tilde T})$ and an arbitrary $w \in \cV(\mathbb{T})$, we then have the Galerkin orthogonality relation
\begin{equation}\label{eq:galerkinorth}
  \norm{ u - w  }^2_B = \norm{u - u_{\mathbb{\tilde T}}}^2_B + \norm{u_{\mathbb{\tilde T}} - w}^2_B \,. 
\end{equation}

In what follows, recalling notation \eqref{eq:functionalcoeffs}, let
\[
    \br(w) = \left( \bigl\langle [B w - f]_\nu , \psi_\lambda \bigr\rangle \right)_{\nu \in \cF,\lambda \in \Theta} \in \ell_2(\cF\times \Theta) \,.
\]
 We denote by $P_{\mathbb{T}}\colon \cV\to \cV$ the orthogonal projection in $\cV$ onto $\cV(\mathbb{T})$, which corresponds to the $V$-orthogonal projection onto $V(\cT_\nu)$ for each $\nu$.

\begin{lemma}\label{lmm:saturation}
   For $w \in \cV(\mathbb{T})$, let $\hat\br \in \ell_2(\cF\times\Theta)$ be such that
   \[
      \norm{ \br(w) - \hat\br }_{\ell_2} \leq \zeta \norm{ \br(w)}_{\ell_2}
   \]
   with $\zeta \in (0, \frac12)$, and let $\tilde\Lambda = (\tilde\Theta_\nu)_{\nu \in \cF}$ with $\tilde \Theta_\nu \subset \Theta$, $\nu \in\cF$, be chosen such that
   \[
         \norm{ \hat\br |_{\tilde\Lambda}}_{\ell_2} \geq \omega_0\norm{\hat\br}_{\ell_2} 
   \]
   with $\omega_0 \in (0,1]$, where $(1+\omega_0)\zeta < \omega_0$. 
   Let $\mathbb{\tilde T} = \msx{\tilde\Lambda}$,
   then with 
   $
    C_{B,\Psi} = {C_\Psi^2 C_B}/ (c_\Psi^2 c_B) \geq 1 ,
    $
    we have
   \[ 
     \norm{u - u_{\mathbb{\tilde T}}}_B \leq \biggl(  1  - \frac{\bigl( \omega_0 - (1 + \omega_0)\zeta \bigr)^2 }{ C_{B,\Psi} }  \biggr)^{\frac12} \norm{ u - w}_B.
    \] 
 \end{lemma}
 
 \begin{proof}
   Using \eqref{eq:framecondition}, we obtain
   \begin{align*}
      \norm{w - u_{\mathbb{\tilde T}}}_B &
      \geq \frac{1}{\sqrt{C_B}} 
       \norm{B (w - u_{\mathbb{\tilde T}})}_{\cV'} 
        \geq \frac{1}{\sqrt{C_B}}  \norm{ P'_{\mathbb{\tilde T}} B (w - u_{\mathbb{\tilde T}}) }_{\cV'}
        = \frac{1}{\sqrt{C_B}}  \norm{ P'_{\mathbb{\tilde T}}  (B w - f) }_{\cV'} \\
        & \geq \frac{1}{\sqrt{C_B} C_\Psi} \Bigl( \sum_\nu \sum_{\lambda \in \Theta} \abs{\langle   [Bw - f]_\nu, P_{\mathbb{\tilde T}} \psi_\lambda \rangle }^2  \Bigr)^{\frac12} \\
        & = \frac{1}{\sqrt{C_B} C_\Psi}  \Bigl( \sum_\nu \sum_{\lambda \in \tilde S_\nu} \abs{\langle   [B w - f]_\nu, \psi_\lambda \rangle }^2  + \sum_\nu \sum_{\lambda \in \Theta \setminus \tilde S_\nu} \abs{\langle   [B w - f]_\nu, P_{\mathbb{\tilde T}} \psi_\lambda \rangle }^2  \Bigr)^{\frac12}  \\
        &\geq \frac{1}{\sqrt{C_B} C_\Psi}  \Bigl( \sum_\nu \sum_{\lambda \in \tilde S_\nu} \abs{\langle   [B w - f]_\nu, \psi_\lambda \rangle }^2  \Bigr)^{\frac12} = \frac{1}{C_\Psi \sqrt{C_B}} \norm{ \br(w) |_{\tilde\Lambda}}.
   \end{align*}        
   We now note that
   \[
      \norm{ \br(w) |_{\tilde\Lambda}} \geq \norm{ \hat\br|_{\tilde\Lambda}} - \norm{ \br(w) - \hat \br} \geq \omega_0\norm{ \hat\br } - \zeta \norm{\br(w)} 
      \geq \omega_0 \norm{ \br(w)} - (1 + \omega_0 ) \zeta \norm{\br(w)} ,
   \]     
   and thus
   \begin{align*}
      \norm{ \br(w) |_{\tilde\Lambda}}  & \geq \bigl( \omega_0 - (1 + \omega_0)\zeta \bigr) \norm{\br(w)}  =   \bigl( \omega_0 - (1 + \omega_0)\zeta \bigr) \Bigl( \sum_\nu \sum_{\lambda \in \Theta} \abs{\langle   [B w - f]_\nu, \psi_\lambda \rangle }^2  \Bigr)^{\frac12} \\
        &\geq   \bigl( \omega_0 - (1 + \omega_0)\zeta \bigr) c_\Psi  \norm{ B w - f}_{\cV'} \geq  \bigl( \omega_0 - (1 + \omega_0)\zeta \bigr) c_\Psi \sqrt{c_B} \norm{ w - u }_B \,.
   \end{align*}
   The statement thus follows with Galerkin orthogonality \eqref{eq:galerkinorth}.
 \end{proof}

\begin{lemma}\label{lm: error reduction}
   Let $\norm{ u_{\mathbb{\tilde T}} - \tilde w }_B \leq \gamma \norm{\hat\br}$. Then
\begin{equation}\label{eq: error reduction}
  \norm{u - \tilde w}_B \leq \delta \norm{ u - w}_B 
\end{equation}
  with
  \[
    \delta = \biggl(  1  - \frac{\bigl( \omega_0 - (1 + \omega_0)\zeta \bigr)^2 }{ C_{B,\Psi} } + \gamma^2 ( 1 + \zeta )^2 C_\Psi^2 C_B  \biggr)^{\frac12}
  \]
  and $C_{B,\Psi}$ from \Cref{lmm:saturation}.
\end{lemma}

\begin{proof}
  Combining
  \[
    \gamma \norm{\hat\br} \leq \gamma ( 1 + \zeta ) \norm{\br(w)} 
     \leq  \gamma ( 1 + \zeta ) C_\Psi \norm{ Bw - f }_{\cV'}
     \leq \gamma ( 1 + \zeta ) C_\Psi \sqrt{C_B} \norm{ w - u }_{B}
  \]
  with the Galerkin orthogonality
  $
    \norm{ u - \tilde w}_B^2 = \norm{ u - u_{\mathbb{\tilde T}}}^2_B + \norm{ u_{\mathbb{\tilde T}} - \tilde w}_B^2$,
    the statement follows.
\end{proof}

Considering \Cref{prop: galerkin accuracy}, there are some requirements for the parameters that enter in $\gamma$ to ensure an error reduction in \Cref{lm: error reduction}. The conditions on $\zeta$ and $\omega_0$ read
\begin{equation}\label{eq:zeta condition}
  0<(C_{B,\Psi}+2)\zeta<1\qquad \text{with}\qquad C_{B,\Psi} = \frac{C_\Psi^2C_B }{c_\Psi^2 c_B}
\end{equation}
and
\begin{equation}\label{eq:omega0 condition}
  \omega_0 > (C_{B,\Psi}+1)\frac{\zeta}{1-\zeta}.
\end{equation}
Then it is possible to also find parameters $\ell$, $\rho$ and $\hat J$ to achieve the condition on $\gamma$
\begin{equation}\label{eq:gamma condition}
  \gamma =  \gamma(\zeta, \ell, \rho, \hat J) < \frac{\omega_0 - (1+\omega_0)\zeta}{(1+\zeta)\sqrt{C_{B,\Psi}C_B}C_\Psi}.
\end{equation}

With the help of \Cref{prop: galerkin accuracy} and \Cref{lm: error reduction} we can thus state the following result on error reduction in each step.

\begin{theorem}\label{thm:errorreduction}
  Assume conditions~\eqref{eq:zeta condition} and~\eqref{eq:omega0 condition}.
  Then there exist parameters $\ell,\rho,\hat J$ such that~\eqref{eq:gamma condition} is satisfied.
  Consequently, the adaptive algorithm achieves the error reduction
  \[
    \norm{u-u^{k+1}}_B \leq \delta \norm{u-u^{k}}_B,
  \]
  where $\delta$ from \Cref{lm: error reduction} is independent of $k$.
\end{theorem}
\begin{proof}
  We first note that \eqref{eq:gamma condition} and 
  \[
  1 - \frac{(\omega_0-(1+\omega_0)\zeta)^2}{C_{B,\Psi}} +  \gamma^2(1+\zeta)^2C_\Psi^2C_B < 1,
  \]
  are equivalent and thus error reduction is achieved due to \eqref{eq: error reduction} in \Cref{lm: error reduction}.
  The terms $2^{-\alpha\ell}$, $\zeta_{\hat J}$ and $\rho$ appearing in $\gamma(\zeta, \ell, \rho, \hat J)$ from \Cref{prop: galerkin accuracy} can be made arbitrarily small for sufficiently large $\hat J$ and $\ell$ and sufficiently small $\rho$. Thus the condition~\eqref{eq:zeta condition} is achievable whenever 
  \[
    \frac{1}{c_\Psi}\frac{1}{\sqrt{c_B}}\left( \frac{\zeta}{1+\zeta} \right) < \frac{\omega_0 - (1+\omega_0)\zeta}{(1+\zeta)\sqrt{C_{B,\Psi}C_B}C_\Psi}, 
  \]
  which directly leads to~\eqref{eq:omega0 condition}, and noting that $\omega_0<1$, the condition ~\eqref{eq:zeta condition} for $\zeta$ follows as well.
\end{proof}

\begin{remark}\label{rem:compComplexity}
  Concerning the computational complexity of the method, note that all substeps of the adaptive method \Cref{alg: adaptive method} have either linear costs up to logarithmic factors in $N(\T)$ or in the target accuracy rate $\eta^{-\frac{1}{s}}$, as seen notably for the most expensive step \textsc{ResEstimate} in \Cref{prop:ResEstimate}. The complexity of \textsc{GalerkinSolve} is linear in $N(\T)$ and only deteriorates for larger $\ell$, as noted in \Cref{rem: complexity gal solve}. For quasi-optimal computational complexity of the adaptive method, it thus only remains to show that the number of degrees of freedom in the approximation that is added in each step of the iterative method remains quasi-optimal.
\end{remark}
\begin{remark}\label{rem:KLgeneralization}
The convergence estimate of \Cref{thm:errorreduction} can be shown in the same manner without requiring the functions $\theta_\mu$ in the expansion of the diffusion cofficient to have wavelet-type multilevel localization properties.
For enforcing the uniform error reduction based on our construction, only piecewise polynomial structure (or approximation) on \emph{some} uniformly refined mesh of each function $\theta_\mu$ in the expansion of the diffusion coefficient is required, which also covers KL-type expansions. 
Specifically, as noted in \Cref{rem:MultilevelStructure}\eqref{rem:MultiLevelErrorReduction}, \Cref{ass:pwpolytheta} can be replaced by only requiring that for a fixed $m$, for each $\mu \in \cM$ there exists $j(\mu)$ such that $\theta_\mu \in \PP_m(\hat\cT_{j(\mu)})$. 
The semidiscrete operator compression in \Cref{prop:operator compression} utilizes the multilevel structure of $(\theta_\mu)_{\mu \in \cM}$, but a similar compression for KL expansions using summability of the norms $\norm{\theta_{\mu}}_{L_\infty}$ was shown in \cite{Gittelson:13}; note, however, that this generally leads to suboptimal complexity.
The required level of additional refinements in \Cref{lmm:spatialrelerr} is still bounded uniformly in terms of $m$, independently of $j(\mu)$. 
The error bound of \Cref{prop:ResEstimate} can again be shown in the same manner using $j(\mu)$. Thus, the convergence result \Cref{thm:errorreduction} follows.
However, our method is designed to exploit the stronger \Cref{ass:pwpolytheta} to obtain the computational complexity bounds summarized in \Cref{rem:compComplexity}.
\end{remark}

\section{Numerical Experiments}
\label{sec:experiments}

The adaptive Galerkin method \Cref{alg: adaptive method} was implemented for spatial dimensions $d=1$ and $d=2$ using the Julia programming language, version 1.9.2. The numerical experiments were performed on a Dell PowerEdge R725 workstation with AMD EPYC 7742 64-core processor. For the code to reproduce the tests, see \cite{ExampleCode}.

\begin{figure}
  \centering
  \begin{minipage}[c]{0.62\textwidth}
    \includegraphics[width  = \linewidth]{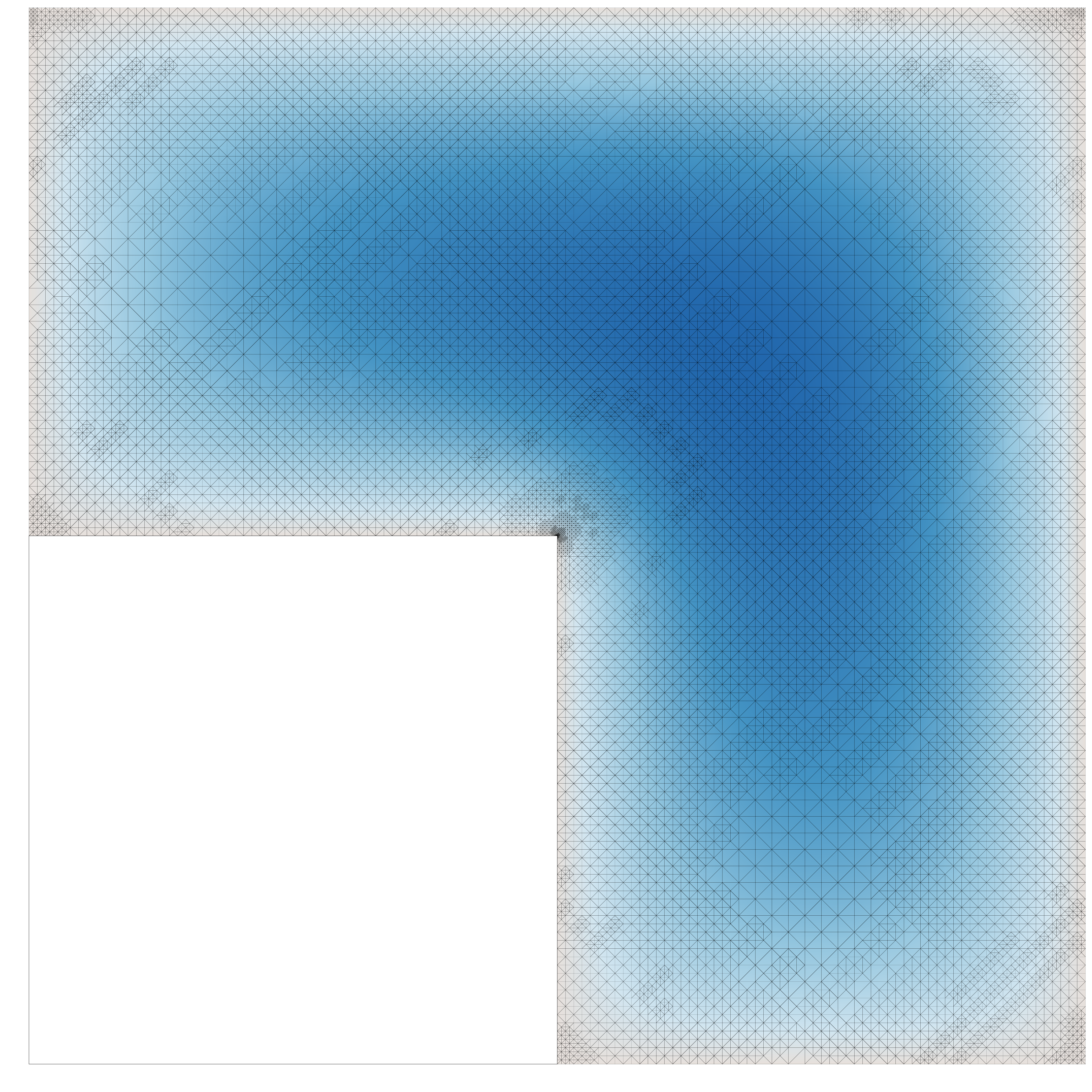}
  \end{minipage}
    \begin{minipage}[c]{0.303\textwidth}
                \includegraphics[width  = \linewidth]{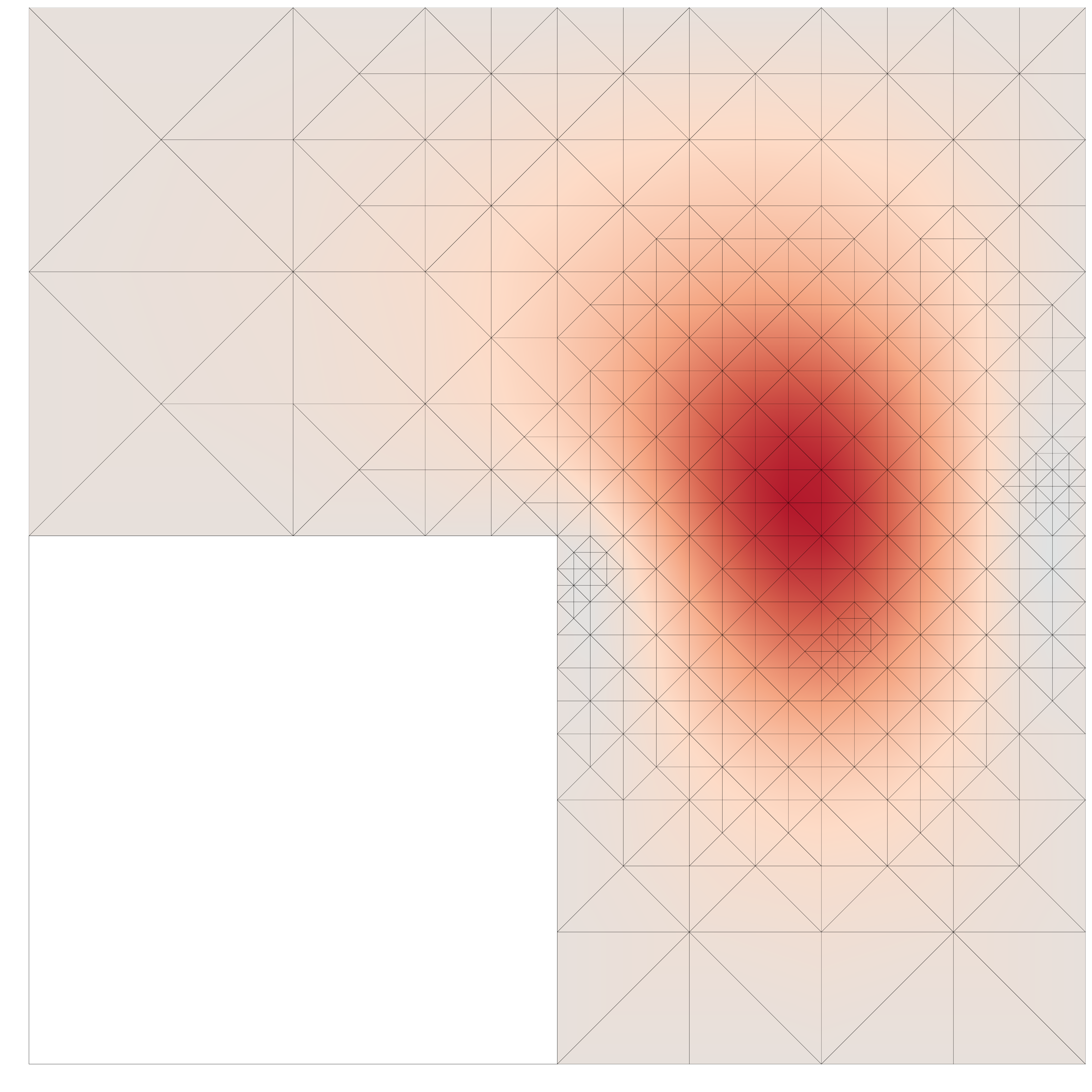}

                \includegraphics[width  = \linewidth]{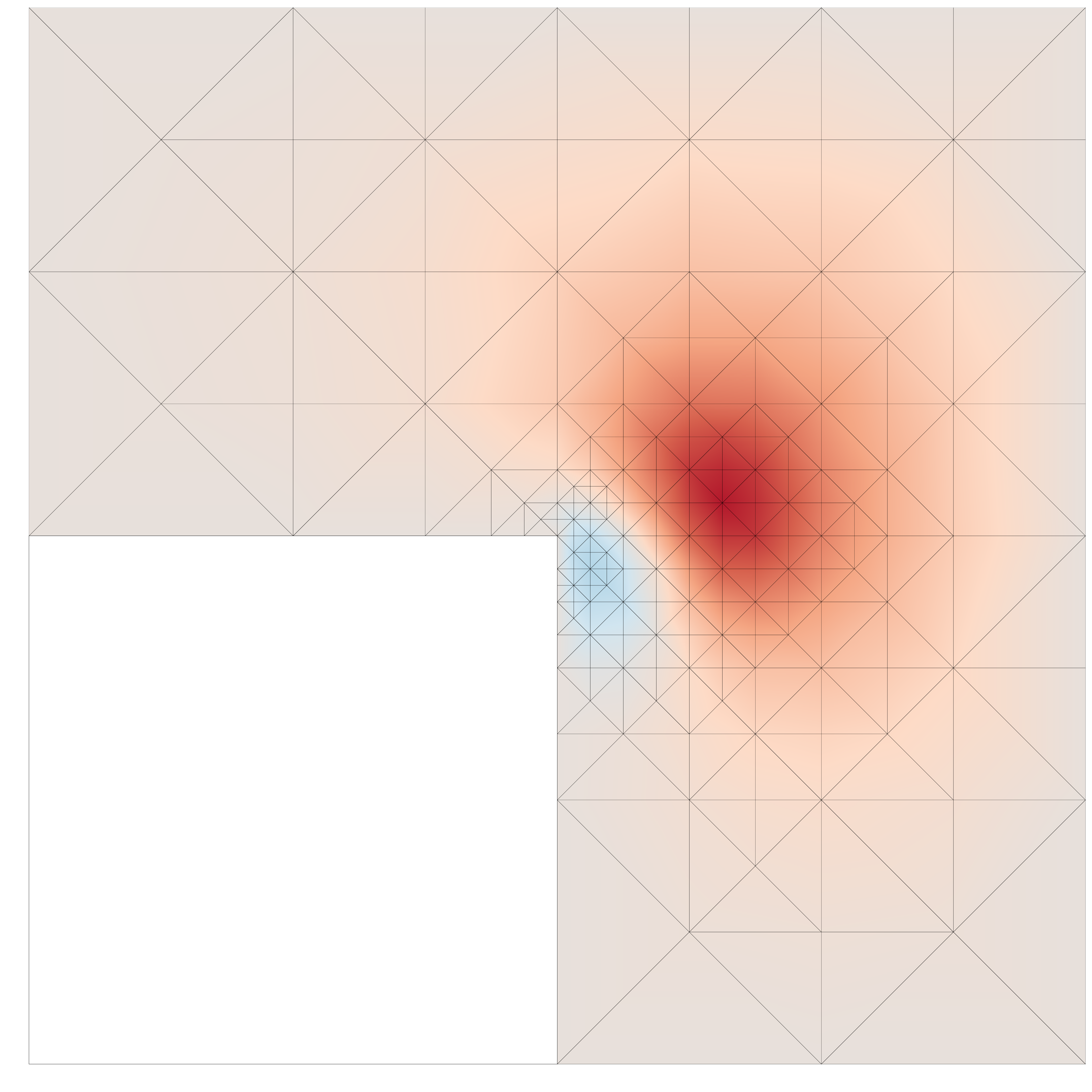}
    \end{minipage}
  \caption{Triangulation of approximations of different components $u_\nu$ in \eqref{eq:legendreexpansion}, analogous to Figure \ref{fig:unu}, generated by the adaptive method for different $\nu$ for $d=2$. Top: $\nu = 0$, bottom: $\nu = e_\mu$ for two different $\mu \in \mathcal{M}$.}
  \label{pic: Triangulation11}
\end{figure}

For $d=2$ we chose the L-shaped domain $D = (0,1)^2\setminus(0, 0.5)^2$, with an initial mesh of 24 congruent triangles with 5 interior nodes. For $d=1$ we simply take $D = (0,1)$. 
In \Cref{pic: Triangulation11} we show refinements of different components that appear during the adaptive method.

For the random fields $a(y)$, we use an expansion in terms of hierarchical hat functions formed by dilations and translations of $\theta(x)=(1-|2x-1|)_+$. Specifically, for $d=1$, $\theta_{\mu}$ with $\mu = (\ell, k)$ is given by 
\begin{equation}\label{eq:theta1d}
\theta_{\ell, k}(x) := c 2^{- \alpha \ell} \theta(2^\ell x - k), \quad k = 0,\ldots,2^\ell-1, \; \ell \in \N_0.
\end{equation}
This yields a wavelet-type multilevel structure satisfying Assumptions \ref{ass:wavelettheta} and thus \eqref{eq:multilevel1} and \eqref{eq:multilevel2}, where
\[
\cM = \{(\ell,k)\colon k = 0,\ldots,2^\ell-1, \; \ell \geq 0 \}
\]
with level parameters $\abs{(\ell,k)} = \ell$.
For $d=2$, we take the isotropic product hierarchical hat functions
\begin{equation}\label{eq:theta2d}
\theta_{\ell, k_1, k_2}(x) := c 2^{- \alpha \ell} \theta(2^\ell x - k_1)\,\theta(2^\ell x - k_2), \quad (\ell,k_1,k_2) \in \cM,
\end{equation}
with 
\begin{multline*}
  \cM = \big\{ \textstyle(\ell,k_1,k_2) \colon \;\ell \in \N_0,\; k_1, k_2 = 0,\frac{1}{2},\ldots,2^\ell-\frac{3}{2},2^\ell-1  \text{ with $ k_1\in \N_0$ or $k_2\in \N_0$},\\\text{and }\supp{\theta_{\ell, k_1, k_2}}\subset D \big\}.  
\end{multline*}

\begin{figure}
	\includegraphics[width=12.5cm]{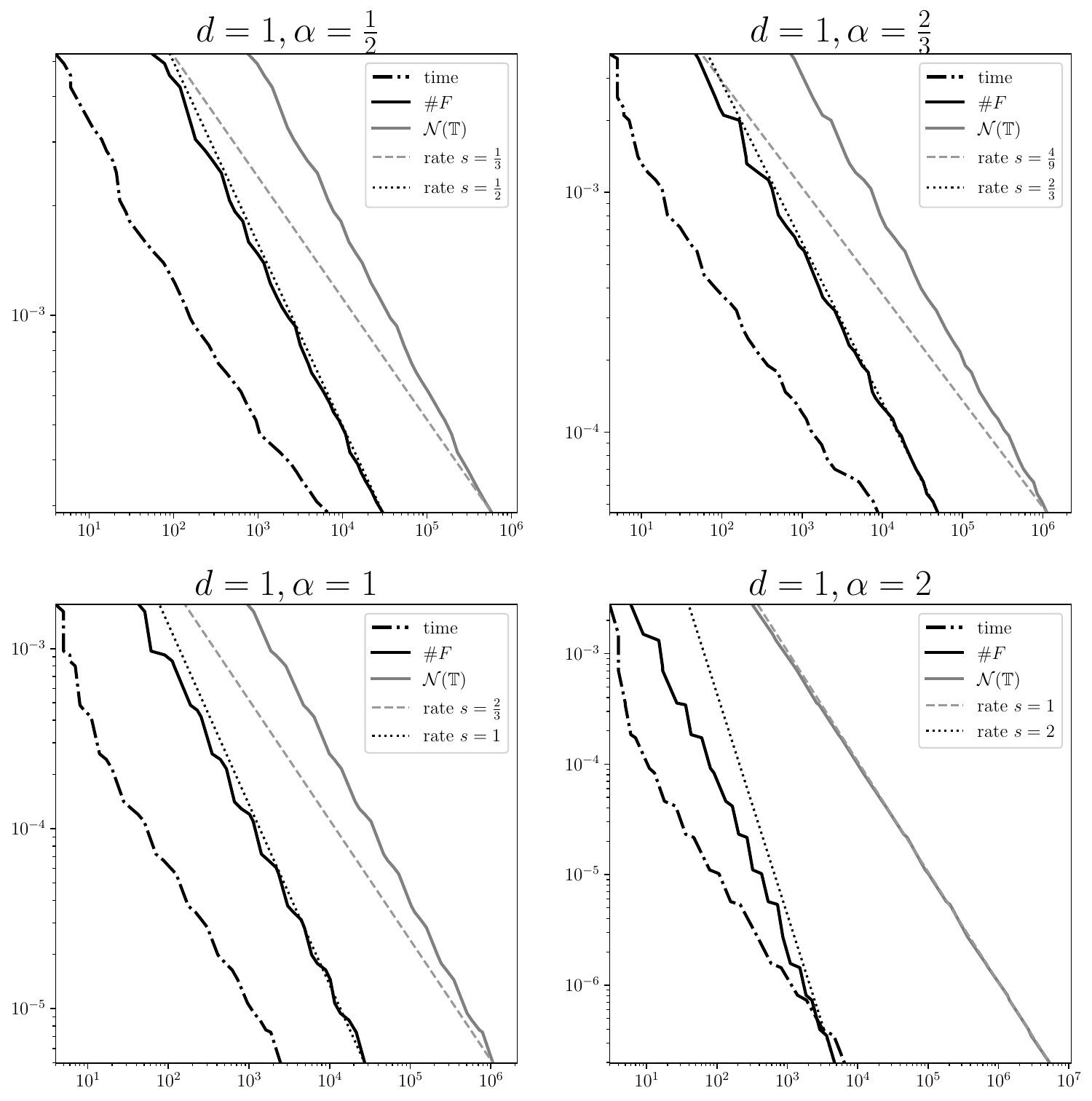}
	\caption{Computed residual bounds for $d=1$ as a function of total number of degrees of freedom $\dim \cV(\T) \eqsim N(\mathbb{T})$ of the current approximation of~$u$ (solid gray lines), degrees of stochastic freedom~$\#F$ (solid black lines), and elapsed computation time in seconds (dash-dotted line).}
	\label{plot:time1d}
\end{figure}
\begin{figure}
	\includegraphics[width=12.5cm]{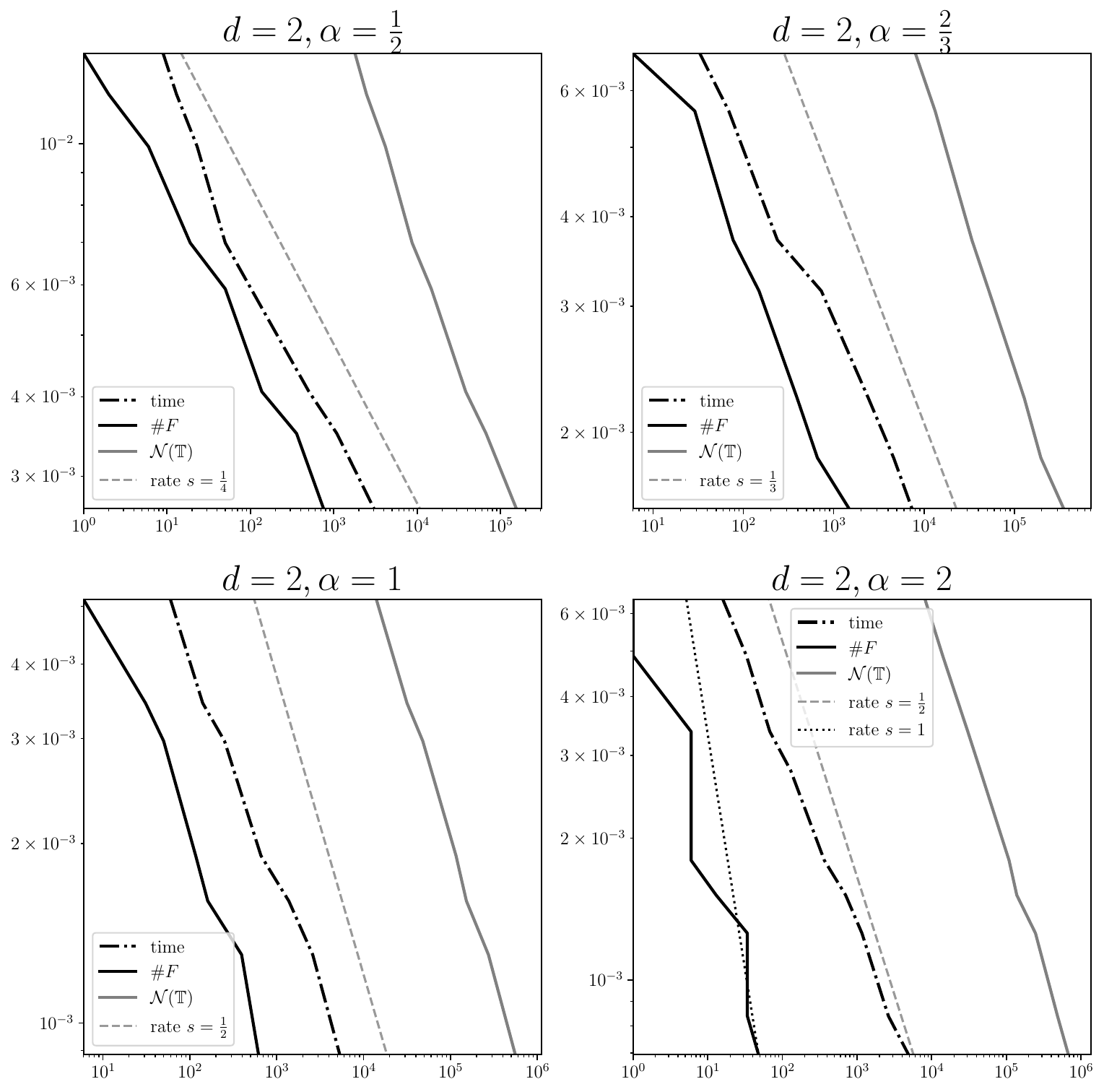}
	\caption{Computed residual bounds for $d=2$ as a function of total number of degrees of freedom $\dim \cV(\T) \eqsim N(\mathbb{T})$ of the current approximation of~$u$ (solid gray lines), number of product Legendre polynomials~$\#F$ (solid black lines), and elapsed computation time in seconds (dash-dotted line).}
	\label{plot:time2d}
\end{figure}

As in~\cite{BV}, we improved the quantitative performance by choosing parameters not strictly to theory. 
The adaptive scheme is tested with $\alpha = \frac12,\frac23,1,2$ for both $d=1$ and $d=2$. We take $f \equiv 1$ and $c= \frac{1}{10}$ in~\eqref{eq:theta1d},~\eqref{eq:theta2d}. The parameters of the scheme are chosen as $\omega_0 = \frac1{10}$, $C_B = \frac{1}{50}$, $c_\Psi = C_\Psi = c_P =C_P = 1$, and the relative accuracy $\rho = \frac{1}{50}$.  To choose an adequate parameter~$\hat J$, we tested different ranges of refinement, until we found no qualitative differences compared to higher values of $\hat J$. This resulted in~$\hat J=2$ for $d= 1$ and $\hat J =1$ for $d=2$.  
The results of the numerical tests are shown in Figure \ref{plot:time1d} for $d=1$ and in Figure \ref{plot:time2d} for $d=2$.
They are compared to the expected approximation rates~\eqref{eq:approxrate} seen as dashed lines. Remarkably, for $d=1$ the rate for degrees of freedom even resembles the expected limiting approximation rates for stochastic variables $\alpha$ instead of $\frac{2}{3}\alpha$. In the case $d=2$, the limiting approximation rates $\frac12 \min\{ \alpha, 1\}$ expected for piecewise linear spatial approximations are recovered by the adaptive method.

\section{Conclusions and Outlook}
\label{sec:conclusion}

We have constructed a novel adaptive stochastic Galerkin finite element method that guarantees a reduction of the error in energy norm in every step of the adaptive scheme while using an independently refined spatial mesh for each product orthonormal polynomial coefficient of the solution. All operations that need to be performed also have costs that are consistent with total computational costs scaling linearly (up to logarithmic factors) with respect to the total number of degrees of freedom $N(\mathbb{T})$. Such scaling of the costs is also visible in the numerical experiments. What is thus left for future work is to show that the number of new degrees of freedom that is added in each iteration of the adaptive scheme is quasi-optimal, which will then yield optimal computational complexity of the method.

As for the earlier wavelet-based method in \cite{BV}, our numerical tests confirm the approximability results \eqref{eq:approxrate} obtained in \cite{BCDS:17}. For $d=1$ with spatial approximation by piecewise linear finite elements, we observe a new effect going beyond the existing tests, where for the reachable accuracies we obtain a better convergence rate than ensured by the theory (and observed for higher-order wavelets in \cite{BV}). This may be related to the particular piecewise linear structures that appear in the exact solution for $d=1$ (see Figure \ref{fig:unu}), but this effect is yet to be understood in detail. 

Finally, let us note that although we have conducted the analysis of the method for the case $d=2$ for simplicity, everything that we have done can be generalized immediately to $d>2$, but the practical implementation in the case $d=3$ poses some further technical challenges.

\bibliographystyle{amsplain}
\bibliography{BEEVadaptsgfem}

\end{document}